\documentclass[a4paper,11pt]{amsart}

\usepackage{mathrsfs}
\usepackage{cite}
\usepackage{amssymb}
\usepackage{graphicx}

\DeclareMathAlphabet{\mathpzc}{OT1}{pzc}{m}{it}

\newtheorem{theorem}{Theorem}[section]
\newtheorem{lemma}[theorem]{Lemma}

\newtheorem{corollary}[theorem]{Corollary}

\newtheorem{claim}[theorem]{Claim}
\theoremstyle{definition}
\newtheorem{definition}[theorem]{Definition}

\newtheorem{assumptions}[theorem]{Assumptions}
\theoremstyle{remark}
\newtheorem{remark}[theorem]{Remark}

\numberwithin{equation}{section}

\allowdisplaybreaks[4]

\begin{document}

\title{On the Milnor fibration of certain\\ Newton degenerate functions} 

\author{Christophe Eyral and Mutsuo Oka}
\address{C. Eyral, Institute of Mathematics, Polish Academy of Sciences, \'Sniadeckich~8, 00-656 Warsaw, Poland}
\email{cheyral@impan.pl}
\address{M. Oka, Department of Mathematics, Tokyo University of Science, 1-3 Kagurazaka, Shinjuku-ku, Tokyo 162-8601, Japan}   
\email{oka@rs.tus.ac.jp}



\begin{abstract}
It is well known that the diffeomorphism-type of the Milnor fibration of a (Newton) non-degenerate polynomial function $f$ is uniquely determined by the Newton boundary of $f$. In the present paper, we generalize this result to certain \emph{degenerate} functions, namely  we show that the diffeomorphism-type of the Milnor fibration of a (possibly degenerate) polynomial function of the form $f=f^1\cdots f^{k_0}$ is uniquely determined by the Newton boundaries of $f^1,\ldots, f^{k_0}$ if $\{f^{k_1}=\cdots=f^{k_m}=0\}$ is a non-degenerate complete intersection variety for any $k_1,\ldots,k_m\in \{1,\ldots, k_0\}$.
\end{abstract}

\maketitle

\markboth{C. Eyral and M. Oka}{On the Milnor fibration of certain Newton degenerate functions}

\section{Introduction}\label{intro}

Let  $f(\mathbf{z})$ and $g(\mathbf{z})$ be two non-constant polynomial functions of $n$ complex variables $\mathbf{z}=(z_1,\ldots,z_n)$ such that $f(\mathbf{0})=g(\mathbf{0})=0$ ($f$ and $g$ may have a \emph{non-isolated} singularity at~$\mathbf{0}$). The goal of this paper is to find ``easy-to-check'' conditions on the functions $f$ and $g$ that guarantee that their Milnor fibrations at $\mathbf{0}$ are isomorphic (i.e., there is a fibre-preserving diffeomorphism from the total space of the Milnor fibration of $f$ onto that of $g$).
In \cite{O1}, the second named author proved that if $f$ and $g$ are (Newton) non-degenerate and have the same Newton boundary, then necessarily they have isomorphic Milnor fibrations (the special cases where, in addition, $f$ is weighted homogeneous or has an isolated singularity at $\mathbf{0}$ were first proved in \cite{O4} and \cite{O3} respectively). The crucial step in the proof of this result is a similar assertion, also proved in \cite{O1}, for $1$-parameter families of functions. It says that if $\tau_0>0$ and if $\{f_t\}_{\vert t\vert\leq \tau_0}$ is a family of non-degenerate polynomial functions with the same Newton boundary, then the Milnor fibrations of $f_t$ and $f_0$ at $\mathbf{0}$ are isomorphic for any $t$, $\vert t\vert\leq \tau_0$. This theorem, in turn, is a consequence of another important result, still proved in \cite{O1}, which asserts that any family $\{f_t\}_{\vert t\vert\leq \tau_0}$ satisfying the above conditions has a so-called ``uniform stable radius'' for the Milnor fibrations of its elements $f_t$. 

Though the scope of the above mentioned theorems is relatively wide, it does not include, for instance, the following quite common situation.  Suppose that $f(\mathbf{z})$ is the product of $k_0\geq 2$ polynomial functions $f^1(\mathbf{z}),\ldots,f^{k_0}(\mathbf{z})$ on $\mathbb{C}^n$ with $n\geq 3$ (so, in particular, we have $\dim_{\mathbf{0}}(V(f^k)\cap V(f^{k'}))\geq n-2\geq 1$, where, as usual, $V(f^k)$ and $V(f^{k'})$ denote the hypersurfaces defined by $f^k$ and $f^{k'}$ respectively; here, the upper index denotes an index, not a power). Then we claim that $f$ is never non-degenerate (and hence the results of \cite{O1} do not apply to this situation). If $f$ is ``convenient'' (i.e., if its Newton boundary intersects each coordinate axis), then our claim is an immediate consequence of a theorem of Kouchnirenko \cite{K} which asserts that a convenient non-degenerate function always has an isolated singularity at the origin. In the above situation, since for $k\not=k'$ the intersection $V(f^k)\cap V(f^{k'})$ is contained in the singular locus of $V(f)$, if the function $f$ is convenient then Kouchnirenko's theorem implies that it must be ``degenerate'' (i.e., not non-degenerate). In the case where $f$ is not a convenient function, our claim follows from a theorem of Bernstein \cite{Bernstein} and Proposition (2.3) in Chapter 4 of \cite{O2} which imply that for $k\not=k'$ the intersection $V(f^k_{\mathbf{w}})\cap V(f^{k'}_{\mathbf{w}})\cap \mathbb{C}^{*n}$ is non-empty whenever $\mathbf{w}\in\mathbb{N}^{*n}$ is such that $f^{k}_{\mathbf{w}}$ and $f^{k'}_{\mathbf{w}}$ are not monomials and the dimension of the Minkowski sum $\Delta(\mathbf{w};f^k)+\Delta(\mathbf{w};f^{k'})$ is $\geq 2$. Here, $\Delta(\mathbf{w};f^k)$ (respectively, $f^k_{\mathbf{w}}$) denotes the face of the Newton polyhedron of $f^k$ (respectively, the face function of $f^k$) with respect to $\mathbf{w}$; similarly for the function $f^{k'}$; see Section~\ref{sect-ndciv} for the definitions. Of course, this implies that the face function $f_{\mathbf{w}}$ of $f$ with respect to $\mathbf{w}$ has a critical point in $V(f_{\mathbf{w}})\cap \mathbb{C}^{*n}$, that is, $f$ is degenerate.

In the present paper, we generalize the results of \cite{O1} to a class of polynomial functions that includes the ``degenerate'' examples mentioned above. A first class of such functions was already given by the authors in \cite{EO} in the case of $1$-parameter families of functions of the form $f_t(\mathbf{z})=f_t^1(\mathbf{z})\cdots f_t^{k_0}(\mathbf{z})$ under a condition called \emph{Newton-admissibility}. This condition says that the Newton boundaries of the functions $f^k_t$ which appear in the product must be independent of $t$ and the (germs at $\mathbf{0}$ of the) varieties $V(f^{k_1}_t,\ldots,f^{k_m}_t):=\{f^{k_1}_t=\cdots=f^{k_m}_t=0\}$ must be non-degenerate, uniformly locally tame, complete intersection varieties for any $k_1,\ldots,k_m\in\{1,\ldots,k_0\}$. The uniform local tameness is a non-degeneracy type condition with respect to the variables corresponding to the ``compact directions'' of the non-compact faces of the Newton polyhedron, the variables corresponding to the ``non-compact directions'' being fixed in a small ball independent of $t$ (for a precise definition, see \cite{EO}).

In fact, under the Newton-admissibility condition, we proved in \cite{EO} a much stronger result on the local geometry of the family of hypersurfaces $V(f_t)$: we showed that any Newton-admissible family is Whitney equisingular and satisfies Thom's condition. Then, as a consequence of these two results, we easily obtained that the Milnor fibrations of $f_t$ and $f_0$ at the origin are isomorphic for all small $t$. 
Note that in the case of non-isolated singularities, the Newton-admissibility condition is a crucial assumption when we want to study geometric properties like Whitney equisingularity or Thom's condition. However, if our goal is only to investigate the Milnor fibrations of the family members $f_t$, then, as we are going to show it in the present work, the uniform local tameness condition (which appears through the Newton-admissibility condition) can be completely dropped.

Our first main theorem here says that if the Newton boundaries of the functions $f^k_t$ $(1\leq k\leq k_0)$ are independent of $t$ and if the varieties $V(f^{k_1}_t,\ldots,f^{k_m}_t)$ are non-degenerate complete intersection varieties for any $k_1,\ldots,k_m\in\{1,\ldots,k_0\}$, then the Milnor fibrations of $f_t$ and $f_0$ at~$\mathbf{0}$ are isomorphic for all small $t$ (see Theorem \ref{fmtotp}). The main step to prove this theorem is the following assertion which is interesting itself. It says that, under the same assumptions, the family $\{f_t\}$ has a uniform stable radius (see Theorem \ref{smt} and Corollary~\ref{cor1}). In the course of the proof of this assertion, we also show how a stable radius for the Milnor fibration of a function of the form $f(\mathbf{z})=f^1(\mathbf{z})\cdots f^{k_0}(\mathbf{z})$ can be obtained when the corresponding varieties $V(f^{k_1},\ldots,f^{k_m})$ are non-degenerate complete intersection varieties for any $k_1,\ldots,k_m\in\{1,\ldots,k_0\}$ (see Theorem \ref{fmt}).

Our second main theorem, which is deduced from the first one, asserts that given two polynomial functions $f(\mathbf{z})=f^1(\mathbf{z})\cdots f^{k_0}(\mathbf{z})$ and $g(\mathbf{z})=g^1(\mathbf{z})\cdots g^{k_0}(\mathbf{z})$, if $V(f^{k_1},\ldots,f^{k_m})$ and $V(g^{k_1},\ldots,g^{k_m})$ are non-degenerate complete intersection varieties for any $k_1,\ldots,k_m\in\{1,\ldots,k_0\}$, and if for each $1\leq k\leq k_0$, the Newton boundaries of $f^k$ and $g^k$ coincide, then the Milnor fibrations of $f$ and $g$ at $\mathbf{0}$ are isomorphic (see Theorem \ref{mtfg}).

Note that in the special case where $k_0=1$ (for which the functions under consideration are necessarily non-degenerate), we recover all the results of \cite{O1}\textemdash a paper from which the present work is inspired.

\section{Non-degenerate complete intersection varieties}\label{sect-ndciv}

Let $\mathbf{z}:=(z_1,\ldots, z_n)$ be coordinates for $\mathbb{C}^n$, and let $f(\mathbf{z})=\sum_\alpha c_\alpha\, \mathbf{z}^\alpha$ be a non-constant polynomial function which vanishes at the origin. Here, $\alpha:=(\alpha_1,\ldots,\alpha_n)\in\mathbb{N}^n$, $c_\alpha\in\mathbb{C}$, and $\mathbf{z}^{\alpha}$ is a notation for the monomial $z_1^{\alpha_1}\cdots z_n^{\alpha_n}$. 
For any $I\subseteq\{1,\ldots, n\}$, we denote by $\mathbb{C}^I$ (respectively, $\mathbb{C}^{*I}$) the set of points $(z_1,\ldots, z_n)\in \mathbb{C}^n$ such that $z_i=0$ if $i\notin I$ (respectively, $z_i=0$ if and only if $i\notin I$).
In particular, we have $\mathbb{C}^{\emptyset}=\mathbb{C}^{*\emptyset}=\{\mathbf{0}\}$ and $\mathbb{C}^{*\{1,\ldots,n\}}=\mathbb{C}^{*n}$, where $\mathbb{C}^*:=\mathbb{C}\setminus \{\mathbf{0}\}$. 
Throughout this paper, we are only interested in a \emph{local} situation, that is, in (arbitrarily small representatives of) germs at the origin.

To start with, let us recall the definition of a non-degenerate complete intersection variety, which is a key notion in this paper (a standard reference for this is \cite{O2}). 

The \emph{Newton polyhedron} $\Gamma_{\! +}(f)$ of the germ of $f$ at the origin $\mathbf{0}\in \mathbb{C}^n$ (with respect to the coordinates $\mathbf{z}=(z_1,\ldots, z_n)$) is the convex hull in $\mathbb{R}_+^n$ of the set
\begin{equation*}
\bigcup_{c_\alpha\not=0} (\alpha+\mathbb{R}_+^n).
\end{equation*}
The \emph{Newton boundary} of $f$ (denoted by $\Gamma(f)$) is the union of the compact faces of  $\Gamma_{\! +}(f)$. 
For any weight vector $\mathbf{w}:=(w_1,\ldots,w_n)\in\mathbb{N}^n$, let $d(\mathbf{w};f)$ be the minimal value of the restriction to $\Gamma_{+}(f)$ of the linear map
\begin{equation*}
\mathbf{x}=(x_1,\ldots,x_n)\in \mathbb{R}^n\mapsto\sum_{i=1}^n x_i w_i \in \mathbb{R},
\end{equation*}
and let $\Delta(\mathbf{w};f)$ be the (possibly non-compact) face of $\Gamma_{+}(f)$ defined by 
\begin{equation*}
\Delta(\mathbf{w};f)=\bigg\{\mathbf{x}\in \Gamma_{+}(f)\, ; \, \sum_{i=1}^n x_i w_i = 
d(\mathbf{w};f)\bigg\}.
\end{equation*} 
Note that if all the $w_i$'s are positive, then $\Delta(\mathbf{w};f)$ is a (compact) face of $\Gamma(f)$, while if $\mathbf{w}=\mathbf{0}$, then $\Delta(\mathbf{w};f)=\Gamma_{+}(f)$. 
The \emph{face function} of $f$ with respect to $\mathbf{w}$ is the function
\begin{equation*}
\mathbf{z}\in\mathbb{C}^n \mapsto \sum_{\alpha\in \Delta(\mathbf{w};f)} c_\alpha\, \mathbf{z}^\alpha\in\mathbb{C}.
\end{equation*} 
Hereafter, this function will be denoted by $f_\mathbf{w}$ or $f_{\Delta(\mathbf{w};f)}$.

Now, consider the set $\mathcal{I}(f)$ consisting of all subsets $I\subseteq \{1,\ldots,n\}$ such that the restriction of $f$ to $\mathbb{C}^I$ (denoted by $f^I$) does not identically vanishes. Clearly, $I\in \mathcal{I}(f)$ if and only if $\Gamma(f^I)=\Gamma(f)\cap \mathbb{R}^I$ is not empty, where $\mathbb{R}^I$ is defined in a similar way as $\mathbb{C}^I$. Hereafter, for any weight vector $\mathbf{w}\in\mathbb{N}^I$, we shall use the simplified following notation:
\begin{equation*}
f^I_{\mathbf{w}}:=(f^I)_{\mathbf{w}}
\quad\mbox{and}\quad
f^I_{\Delta(\mathbf{w};f^I)}:=(f^I)_{\Delta(\mathbf{w};f^I)}.
\end{equation*}
(Of course, $\mathbb{N}^I$ is defined in a similar way as $\mathbb{C}^I$ and $\mathbb{R}^I$.)
Note that for all $\mathbf{w}\in\mathbb{N}^I$, we have
\begin{equation*}
f^I_{\mathbf{w}}\equiv f^I_{\Delta(\mathbf{w};f^I)} = f_{\Delta(\mathbf{w};f^I)}.
\end{equation*}

\begin{definition}[see \cite{K}]
The germ at $\mathbf{0}$ of the hypersurface $V(f):=f^{-1}(0)\subseteq \mathbb{C}^n$ is called \emph{non-degenerate} if for any ``positive'' weight vector $\mathbf{w}\in\mathbb{N}^{*n}$ (i.e., $w_i>0$ for all $i$), the hypersurface 
\begin{equation*}
V^*(f_{\mathbf{w}}):=\{\mathbf{z}\in\mathbb{C}^{*n}\mid 
f_{\mathbf{w}}(\mathbf{z})=0\}
\end{equation*}
is a reduced, non-singular, hypersurface in the complex torus $\mathbb{C}^{*n}$. This means that $f_{\mathbf{w}}$ has no critical point in $V^*(f_\mathbf{w})$, that is, the $1$-form $df_{\mathbf{w}}$ is nowhere vanishing in $V^*(f_\mathbf{w})$.
We emphasize that $V^*(f_{\mathbf{w}})$ is globally defined in~$\mathbb{C}^{*n}$.
\end{definition}

Now, consider $k_0$ non-constant polynomial functions $f^1(\mathbf{z}),\ldots, f^{k_0}(\mathbf{z})$ which all vanish at the origin. 

\begin{definition}[see \cite{O2}]
We say that the germ at $\mathbf{0}$ of the variety 
\begin{equation*}
V(f^1,\ldots,f^{k_0}):=\{\mathbf{z}\in\mathbb{C}^n\mid f^1(\mathbf{z}) = 
\cdots = f^{k_0}(\mathbf{z})=0\}
\end{equation*}
is a germ of a \emph{non-degenerate complete intersection variety} if for any positive weight vector $\mathbf{w}\in\mathbb{N}^{*n}$, the variety
\begin{equation*}
V^*(f^1_{\mathbf{w}},\ldots,f^{k_0}_{\mathbf{w}}):=\{\mathbf{z}\in\mathbb{C}^{*n}\mid f^1_{\mathbf{w}}(\mathbf{z}) = \cdots = f^{k_0}_{\mathbf{w}}(\mathbf{z})=0\}
\end{equation*}
is a reduced, non-singular, complete intersection variety in $\mathbb{C}^{*n}$, that is, the $k_0$-form 
\begin{equation*}
df^1_{\mathbf{w}}\wedge \cdots\wedge df^{k_0}_{\mathbf{w}}
\end{equation*}
is nowhere vanishing in  $V^*(f^1_{\mathbf{w}},\ldots,f^{k_0}_{\mathbf{w}})$. Again, we emphasize that $V^*(f^1_{\mathbf{w}},\ldots,f^{k_0}_{\mathbf{w}})$ is globally defined in~$\mathbb{C}^{*n}$.
\end{definition}

\begin{remark}\label{rem-rI}
If $V(f^1,\ldots,f^{k_0})$ is a germ of a non-degenerate complete intersection variety, then, by \cite[Chapter III, Lemma (2.2)]{O2}, for any $I\in\mathcal{I}(f^1)\cap\cdots\cap \mathcal{I}(f^{k_0})$, the germ at~$\mathbf{0}$ of the variety 
\begin{equation*}
V^I(f^1,\ldots,f^{k_0}):=\{\mathbf{z}\in\mathbb{C}^I\mid f^{1,I}(\mathbf{z}) = 
\cdots = f^{k_0,I}(\mathbf{z})=0\}
\end{equation*}
is a germ of a non-degenerate complete intersection variety too. In other words, for any $\mathbf{w}\in\mathbb{N}^{*I}$, the $k_0$-form 
$
df^{1,I}_{\mathbf{w}}\wedge \cdots\wedge df^{k_0,I}_{\mathbf{w}}
$
is nowhere vanishing in 
\begin{equation*}
V^{*I}(f^1_{\mathbf{w}},\ldots,f^{k_0}_{\mathbf{w}}):=\{\mathbf{z}\in\mathbb{C}^{*I}\mid f^{1,I}_{\mathbf{w}}(\mathbf{z}) = \cdots = f^{k_0,I}_{\mathbf{w}}(\mathbf{z})=0\}.
\end{equation*}
(As usual, $f^{k,I}$ is the restriction of $f^k$ to $\mathbb{C}^I$ and $f^{k,I}_{\mathbf{w}}$ is the face function $(f^{k,I})_{\mathbf{w}}\equiv (f^{k,I})_{\Delta(\mathbf{w};f^{k,I})}$.)
\end{remark}

\section{Stable radius for the Milnor fibration}\label{sect-srfmf}

Let again $f^1(\mathbf{z}),\ldots, f^{k_0}(\mathbf{z})$ be non-constant polynomial functions of $n$ complex variables $\mathbf{z}=(z_1,\ldots,z_n)$ such that $f^k(\mathbf{0})=0$ for all $1\leq k\leq k_0$. 

\begin{assumptions}\label{ass-srfmf}
Throughout this section, we assume that for any $k_1,\ldots, k_m\in\{1,\ldots, k_0\}$, the germ of the variety $V(f^{k_1},\ldots,f^{k_m})$ at $\mathbf{0}$ is the germ of a non-degenerate complete intersection variety. 
\end{assumptions}

We start with the following lemma which is crucial for the paper. Note that in the special case where $k_0=1$, the function $f^1$ (or the hypersurface $V(f^1)$) is non-degenerate, and the lemma below coincides with Lemma 1 of \cite{O1}.

\begin{lemma}\label{Lemma1}
Under Assumptions \ref{ass-srfmf}, there exists $\varepsilon>0$ such that for any $k_1,\ldots, k_m\in\{1,\ldots, k_0\}$, any $I\subseteq\{1,\ldots,n\}\,$ with $I\in\mathcal{I}(f^{k_1})\cap\cdots\cap\mathcal{I}(f^{k_m})$, any weight vector $\mathbf{w}=(w_1,\ldots,w_n)\in\mathbb{N}^{I}$ and any (possibly zero) $\lambda\in \mathbb{C}$,  if $\mathbf{a}=(a_1,\ldots,a_n)$ is a point in $\mathbb{C}^{I}$ satisfying the following two conditions:
\begin{enumerate}
\item
$f_\mathbf{w}^{k_1,I}(\mathbf{a})=\cdots=f_\mathbf{w}^{k_m,I}(\mathbf{a})=0$;
\item
there exists a $m$-tuple $(\mu_{k_1},\ldots,\mu_{k_m})\in\mathbb{C}^m\setminus\{\mathbf{0}\}$ such that for all $i\in I$:
\begin{equation*}
\sum_{j=1}^m \mu_{k_j}\frac{\partial f_{\mathbf{w}}^{k_j,I}}{\partial z_{i}}(\mathbf{a})=
\left\{
\begin{aligned}
& \lambda\bar a_{i} &&\mbox{if} && i\in I\cap I(\mathbf{w}),\\
& 0 &&\mbox{if} && i\in I\setminus I(\mathbf{w}),
\end{aligned}
\right.
\end{equation*}
where $\bar a_i$ is the complex conjugate of $a_i$ and $I(\mathbf{w}):=\{i\in \{1,\ldots,n\};\, w_i=0\}$; 
\end{enumerate}
then we must have
\begin{equation*}
\mathbf{a}\notin\bigg\{\mathbf{z}\in \mathbb{C}^{*I}\, ;\, \sum_{i\in I\cap I(\mathbf{w})}|z_i|^2\leq \varepsilon^2\bigg\}.
\end{equation*}
\end{lemma}

We shall prove Lemma \ref{Lemma1} at the end of this section. First, let us use it in order to prove the following first important theorem.

\begin{theorem}\label{fmt}
Under Assumptions \ref{ass-srfmf}, if $f(\mathbf{z}):=f^1(\mathbf{z})\cdots f^{k_0}(\mathbf{z})$, then the number $\varepsilon$ which appears in Lemma \ref{Lemma1} is a stable radius for the Milnor fibration of $f$.
\end{theorem}

We recall that $\varepsilon$ is called a \emph{stable radius} for the Milnor fibration of $f$ if for any $0<\varepsilon_1\leq \varepsilon_2<\varepsilon$, there exists $\delta(\varepsilon_1,\varepsilon_2)>0$ such that for any $\eta\in\mathbb{C}$ with $0<|\eta|\leq \delta(\varepsilon_1,\varepsilon_2)$, the hypersurface $f^{-1}(\eta)\subseteq\mathbb{C}^n$ is non-singular in $\mathring{B}_{\varepsilon}:=\{\mathbf{z}\in\mathbb{C}^n\, ;\, \Vert \mathbf{z} \Vert<\varepsilon\}$ and transversely intersects the spheres $S_{\varepsilon_{12}}:=\{\mathbf{z}\in\mathbb{C}^n\, ;\, \Vert \mathbf{z} \Vert=\varepsilon_{12}\}$ for any $\varepsilon_1\leq \varepsilon_{12}\leq\varepsilon_2$. 
The existence of such a radius was proved by H. A. Hamm and L\^e D\~ung Tr\'ang in \cite[Lemme (2.1.4)]{HL}.  

Note that Theorem \ref{fmt} includes Theorem 1 of \cite{O1} which is obtained by taking $k_0=1$.

\begin{proof}[Proof of Theorem \ref{fmt}]
We argue by contradiction. By \cite[Corollary 2.8]{Milnor}, for $\delta>0$ small enough, the fibres $f^{-1}(\eta)\cap\mathring{B}_\varepsilon$ are non-singular for any $\eta$, $0<|\eta|\leq\delta$. It follows that if the assertion in Theorem \ref{fmt} is not true, then, by the Curve Selection Lemma (see \cite{Milnor,Hamm}), there exists a real analytic curve $\mathbf{z}(s)=(z_1(s),\ldots,z_n(s))$ in $\mathbb{C}^n$, $0\leq s\leq 1$, and a family of complex numbers $\lambda(s)$, $0<s\leq 1$, satisfying the following three conditions:
\begin{enumerate}
\item[(i)]
$\frac{\partial f}{\partial z_i}(\mathbf{z}(s)) = \lambda(s) \bar z_i(s)$ for $1\leq i\leq n$ and $s\not=0$;
\item[(ii)]
$f(\mathbf{z}(0))=0$ but $f(\mathbf{z}(s))$ is not constantly zero;
\item[(iii)]
there exists $\varepsilon'>0$ such that $\varepsilon'\leq \Vert \mathbf{z}(s)\Vert\leq \varepsilon$.
\end{enumerate}
Note that, by (i) and (ii), $\lambda(s)\not\equiv 0$ and we can express it in a Laurent series
\begin{equation*}
\lambda(s)=\lambda_0 s^c+\cdots,
\end{equation*}
where $\lambda_0\in\mathbb{C}^*$. Throughout, the dots ``$\cdots$'' stand for the higher order terms. Let $I:=\{i\, ;\, z_i(s)\not\equiv 0\}$. 
By (ii), $I\in\mathcal{I}(f)$, and hence $I\in\mathcal{I}(f^{1})\cap\cdots\cap\mathcal{I}(f^{k_0})$. For each $i\in I$, consider the Taylor expansion
\begin{equation*}
z_i(s)=a_i s^{w_i}+\cdots,
\end{equation*}
where $a_i\in\mathbb{C}^*$ and $w_i\in\mathbb{N}$. 

\begin{claim}\label{claim1}
There exists $1\leq k\leq k_0$ such that $f^{k,I}_{\mathbf{w}}(\mathbf{a})\equiv(f^{k,I})_{\mathbf{w}}(\mathbf{a})=0$, where $\mathbf{a}$ and $\mathbf{w}$ are the points in $\mathbb{C}^{*I}$ and $\mathbb{N}^{I}$, respectively, whose $i$th coordinates ($i\in I$) are $a_{i}$ and $w_{i}$ respectively.
\end{claim}

Hereafter, to simplify the notation, we shall assume that $I=\{1,\ldots,n\}$, so that the function $f^{k,I}$ is simply written as $f^k$, the intersection $I\cap I(\mathbf{w})$ is written as $I(\mathbf{w})$ (where, as in Lemma \ref{Lemma1}, $I(\mathbf{w})$ is the set of all indexes $i\in \{1,\ldots,n\}$ for which $w_i=0$), and so on. The argument for a general $I$ is completely similar.

Before proving Claim \ref{claim1}, let us first complete the proof of Theorem \ref{fmt}.
For that purpose, we look at the set consisting of all integers $k$ for which $f^{k}_{\mathbf{w}}(\mathbf{a})=0$, which is not empty by Claim \ref{claim1}. For simplicity again, we shall assume: 
\begin{equation*}
\begin{aligned}
& f^{k}_{\mathbf{w}}(\mathbf{a})=0  &&\mbox{ for }  && 1\leq k\leq k_0'\leq k_0;\\
& f^{k}_{\mathbf{w}}(\mathbf{a})\not=0  &&\mbox{ for }  && k'_0+1\leq k\leq k_0.
\end{aligned}
\end{equation*}
Write $f=f^1\cdots f^{k_0'}\cdot h$, where $h:=f^{k_0'+1}\cdots f^{k_0}$ if $k_0'\leq k_0-1$ and $h:=1$ if $k_0'=k_0$. Then for all $1\leq i\leq n$, we have
\begin{equation}\label{e1}
\frac{\partial f}{\partial z_i}(\mathbf{z}(s)) = \sum_{k=1}^{k_0'}
\bigg( \frac{\partial f^{k}}{\partial z_i}(\mathbf{z}(s))\cdot h(\mathbf{z}(s))\cdot\prod_{\genfrac{}{}{0pt}{}{1\leq \ell\leq k_0'}{\ell\not=k}} f^{\ell}(\mathbf{z}(s))\bigg)
+\frac{\partial h}{\partial z_i}(\mathbf{z}(s))\cdot\prod_{1\leq k\leq k_0'} f^{k}(\mathbf{z}(s)).
\end{equation}
For each $1\leq k\leq k_0'$, if $o_k\equiv\mbox{ord}\, f^{k}(\mathbf{z}(s))$ denotes the order (in $s$) of $f^{k}(\mathbf{z}(s))$ and if $e_k:=d(\mathbf{w};f^{k})-o_k+\sum_{\ell=1}^{k_0'} o_{\ell}$, then 
\begin{equation}\label{e2}
\mbox{ord}\bigg( \frac{\partial f^{k}}{\partial z_i}(\mathbf{z}(s))\cdot h(\mathbf{z}(s))\cdot\prod_{\genfrac{}{}{0pt}{}{1\leq \ell\leq k_0'}{\ell\not=k}} f^{\ell}(\mathbf{z}(s))\bigg)\geq
d(\mathbf{w};h)-w_i+e_k,
\end{equation}
and the equality holds if and only if $\frac{\partial f^{k}_{\mathbf{w}}}{\partial z_i}(\mathbf{a})\not=0$.
Since $o_k>d(\mathbf{w};f^{k})$ for $1\leq k\leq k'_0$, we also have
\begin{equation}\label{e3}
\mbox{ord}\bigg(\frac{\partial h}{\partial z_i}(\mathbf{z}(s))\cdot\prod_{\ell=1}^{k_0'} f^{\ell}(\mathbf{z}(s))\bigg) \geq d(\mathbf{w};h)-w_i+\sum_{\ell=1}^{k_0'} o_{\ell} > 
d(\mathbf{w};h)-w_i+e_k
\end{equation}
for all $1\leq k\leq k_0'$.
Still for simplicity, let us assume that
\begin{equation*}
e_{\mbox{\tiny min}}:=e_1=\cdots=e_{k_0''}<e_{k_0''+1}\leq \cdots\leq e_{k_0'}.
\end{equation*}
The relations \eqref{e1}, \eqref{e2} and \eqref{e3} show that there exist $\mu_1,\ldots,\mu_{k''_0}\in\mathbb{C}^*$ such that for any $1\leq i\leq n$:
\begin{equation*}
\frac{\partial f}{\partial z_{i}}(\mathbf{z}(s))=\sum_{k=1}^{k''_0}\frac{\partial {f^{k}_{\mathbf{w}}}}{\partial z_{i}} (\mathbf{a})\cdot \mu_k\cdot s^{d(\mathbf{w};h)-w_{i}+e_{\mbox{\tiny min}}}+\cdots,
\end{equation*}
and hence, by multiplying both sides of the relation (i) by $s^{w_i}$: 
\begin{equation}\label{ibciol}
\sum_{k=1}^{k''_0}\frac{\partial {f^{k}_{\mathbf{w}}}}{\partial z_{i}} (\mathbf{a})\cdot \mu_k\cdot s^{d(\mathbf{w};h)+e_{\mbox{\tiny min}}}+\cdots = 
\lambda_0 \bar a_{i}s^{c+2w_{i}}+\cdots.
\end{equation}
Note that the coefficient $\lambda_0 \bar a_i$ of $s^{c+2w_i}$ on the right-hand side of \eqref{ibciol} being non-zero, we must have $d(\mathbf{w};h)+e_{\mbox{\tiny min}}\leq c+2w_i$ for any $1\leq i\leq n$, and since $I(\mathbf{w})\not=\emptyset$ (by (iii)), in fact we have $d(\mathbf{w};h)+e_{\mbox{\tiny min}}\leq c$. It follows that for any $i\notin I(\mathbf{w})$, the sum
\begin{equation*}
S_i:=\sum_{k=1}^{k''_0} \mu_k \frac{\partial {f^{k}_{\mathbf{w}}}}{\partial z_{i}} (\mathbf{a})
\end{equation*} 
 vanishes. (Indeed, if there exists $i_0\notin I(\mathbf{w})$ such that $S_{i_0}\not=0$, then 
$
c+2w_{i_0}=d(\mathbf{w};h)+e_{\mbox{\tiny min}}\leq c, 
$
which is a contradiction.) Now, if we also have $S_i=0$ for all $i\in I(\mathbf{w})$, then the condition (2) of Lemma \ref{Lemma1} is satisfied. (Note that the complex number denoted by $\lambda$ in Lemma \ref{Lemma1} may vanish.) However, the relation (iii) implies
\begin{equation}\label{contra-a}
\mathbf{a}\in\bigg\{\mathbf{z}\in \mathbb{C}^{*n};\sum_{i\in I(\mathbf{w})}|z_i|^2\leq \varepsilon^2\bigg\},
\end{equation}
which contradicts the conclusion of this lemma.
If there exists $i_0\in I(\mathbf{w})$ such that $S_{i_0}\not=0$, then it follows that
$S_i\not=0$ for any $i\in I(\mathbf{w})$, so that for all such $i$'s:
\begin{equation*}
S_i\equiv\sum_{k=1}^{k''_0} \mu_k \frac{\partial {f^{k}_{\mathbf{w}}}}{\partial z_i} (\mathbf{a}) = \lambda_0 \bar a_i.
\end{equation*}
Thus the condition (2) of Lemma \ref{Lemma1} is satisfied in this case too, and again the relation (iii) (which implies \eqref{contra-a}) leads to a contradiction with the conclusion of the lemma. So, up to Claim \ref{claim1}, the theorem is proved.

To complete the proof of the theorem, it remains to prove Claim \ref{claim1}.

\begin{proof}[Proof of Claim \ref{claim1}]
Again, to simplify, we assume $I=\{1,\ldots,n\}$, so that $f^{k,I}=f^k$, $I\cap I(\mathbf{w})=I(\mathbf{w})$, etc. We argue by contradiction. Suppose $f^{k}_{\mathbf{w}}(\mathbf{a})\not=0$ for all $1\leq k\leq k_0$. Then  $d(\mathbf{w};f^{k})=o_k$ for all $1\leq k\leq k_0$, where $o_k$ is the order of $f^{k}(\mathbf{z}(s))$. Also, note that, by (ii), there exists $1\leq k_1\leq k_0$ such that $f^{k_1}(\mathbf{z}(0))=0$. If $I(\mathbf{w})=\{1,\ldots,n\}$, then $d(\mathbf{w};f^{k_1})=0$ and
\begin{equation*}
f^{k_1}(\mathbf{z}(s))=f^{k_1}_{\mathbf{w}}(\mathbf{a})\, s^0+\cdots, 
\end{equation*}
and therefore $0=f^{k_1}(\mathbf{z}(0))=f^{k_1}_{\mathbf{w}}(\mathbf{a})$, which is a contradiction. So, from now on, suppose that $I(\mathbf{w})$ is a proper subset of $\{1,\ldots,n\}$ and $d(\mathbf{w};f^{k_1})\not=0$.
Put $e:=\sum_{k=1}^{k_0} o_{k}$. 
Then, as above, there exist non-zero complex numbers $\mu_1,\ldots,\mu_{k_0}$  (actually, here, for each $k$, $\mu_k=\prod_{\ell\not= k}f^{\ell}_{\mathbf{w}}(\mathbf{a})$) such that for any $1\leq i\leq n$:
\begin{equation*}
\frac{\partial f}{\partial z_{i}}(\mathbf{z}(s))=\sum_{k=1}^{k_0}\frac{\partial {f^{k}_{\mathbf{w}}}}{\partial z_{i}} (\mathbf{a})\cdot \mu_k\cdot s^{-w_{i}+e}+\cdots,
\end{equation*}
and hence, by multiplying both sides of the relation (i) by $s^{w_i}$:
\begin{equation*}
\sum_{k=1}^{k_0}\frac{\partial {f^{k}_{\mathbf{w}}}}{\partial z_{i}} (\mathbf{a})\cdot \mu_k\cdot s^{e}+\cdots = 
\lambda_0 \bar a_{i}s^{c+2w_{i}}+\cdots.
\end{equation*}
Again, since $\lambda_0 \bar a_i\not=0$ and $I(\mathbf{w})\not=\emptyset$, we have $e\leq c$ and the sum $\sum_{k=1}^{k_0}\mu_k\frac{\partial {f^{k}_{\mathbf{w}}}}{\partial z_{i}} (\mathbf{a})$ vanishes for all $i\in I(\mathbf{w})^c:=\{1,\ldots,n\}\setminus I(\mathbf{w})$. As $f^{k}_{\mathbf{w}}$ is weighted homogeneous, this, together with the Euler identity, implies that 
\begin{align*}
0 & = \sum_{i\in I(\mathbf{w})^c}a_i w_i\bigg(\underbrace{\sum_{k=1}^{k_0}\frac{\partial {f^{k}_{\mathbf{w}}}}{\partial z_{i}} (\mathbf{a})\cdot \mu_k}_{=0}\bigg) = 
\sum_{k=1}^{k_0}\bigg(\prod_{\genfrac{}{}{0pt}{}{1\leq \ell\leq k_0}{\ell\not=k}}f^{\ell}_{\mathbf{w}}(\mathbf{a})\cdot \sum_{i\in I(\mathbf{w})^c}a_i w_i \frac{\partial {f^{k}_{\mathbf{w}}}}{\partial z_{i}} (\mathbf{a})\bigg)\\
& = \sum_{k=1}^{k_0}\bigg(\prod_{\genfrac{}{}{0pt}{}{1\leq \ell\leq k_0}{\ell\not=k}}f^{\ell}_{\mathbf{w}}(\mathbf{a})\bigg)\cdot d(\mathbf{w};f^{k})\cdot f^{k}_{\mathbf{w}}(\mathbf{a})
= \bigg(\prod_{\ell=1}^{k_0}f^{\ell}_{\mathbf{w}}(\mathbf{a})\bigg)\cdot\sum_{k=1}^{k_0} d(\mathbf{w};f^{k})\not=0,
\end{align*}
which is a contradiction too.
\end{proof}
This completes the proof of Theorem \ref{fmt} (up to Lemma \ref{Lemma1}).
\end{proof}

Now let us prove Lemma \ref{Lemma1}.

\begin{proof}[Proof of Lemma \ref{Lemma1}]
First, observe that if the assertion fails for some $k_1,\dots,k_m$, $I$ and $\Delta(\mathbf{w};f^{k_1}),\ldots,\Delta(\mathbf{w};f^{k_m})$ such that $I\cap I(\mathbf{w})=\emptyset$, then for any $\varepsilon>0$ the set 
\begin{align*}
\bigg\{\mathbf{z}\in \mathbb{C}^{*I}\, ;\, \sum_{i\in I\cap I(\mathbf{w})=\emptyset}|z_i|^2\leq \varepsilon^2\bigg\}
\end{align*}
is nothing but $\mathbb{C}^{*I}$ and there exists a point $\mathbf{a}$ in it that satisfies the conditions (1) and (2) of the lemma; in particular, $\mathbf{a}\in V^{*I}(f^{k_1}_{\mathbf{w}},\ldots,f^{k_m}_{\mathbf{w}})$ and the vectors $\mathbf{z}^{k_{1}}(\mathbf{a}),\ldots,\mathbf{z}^{k_{m}}(\mathbf{a})\in \mathbb{C}^I$ whose $i$th coordinates ($i\in I$) are 
\begin{equation*}
\frac{\partial f_{\mathbf{w}}^{k_{1},I}}{\partial z_{i}}(\mathbf{a}),\ldots,\frac{\partial f_{\mathbf{w}}^{k_{m},I}}{\partial z_{i}}(\mathbf{a}),
\end{equation*}
 respectively, are linearly dependent, that is,
\begin{align*}
df^{k_1,I}_{\mathbf{w}}(\mathbf{a})\wedge\cdots\wedge df^{k_m,I}_{\mathbf{w}}(\mathbf{a})=0.
\end{align*}
However, since $I\in\mathcal{I}(f^{k_1})\cap\cdots\cap\mathcal{I}(f^{k_m})$ and $I\cap I(\mathbf{w})=\emptyset$, this contradicts Assumptions \ref{ass-srfmf} which imply
\begin{align*}
df^{k_1,I}_{\mathbf{w}}(\mathbf{p})\wedge\cdots\wedge df^{k_m,I}_{\mathbf{w}}(\mathbf{p})\not=0
\end{align*}
for any $\mathbf{p}\in V^{*I}(f^{k_1}_{\mathbf{w}},\ldots,f^{k_m}_{\mathbf{w}})$ (see Remark \ref{rem-rI}).

Now, assume that the assertion in Lemma \ref{Lemma1} fails for some $k_1,\dots,k_m$, $I$ and 
$\Delta(\mathbf{w};f^{k_1}),\ldots,\allowbreak\Delta(\mathbf{w};f^{k_m})$ 
such that $I\cap I(\mathbf{w})\not=\emptyset$. 
Again, without loss of generality, and in order to simplify the notation, we assume that $I=\{1,\ldots,n\}$, so that $f^{k,I}_{\mathbf{w}}=f^{k}_{\mathbf{w}}$, $I\cap I(\mathbf{w})=I(\mathbf{w})$, $\mathbb{C}^{*I}=\mathbb{C}^{*n}$, etc.
Then there is a sequence $\{\mathbf{p}_q\}_{q\in\mathbb{N}}$ of points in $\mathbb{C}^{*n}$ and a sequence $\{\lambda_q\}_{q\in\mathbb{N}}$ of complex numbers such that:
\begin{enumerate}
\item
$f_\mathbf{w}^{k_1}(\mathbf{p}_q)=\cdots=f_\mathbf{w}^{k_m}(\mathbf{p}_q)=0$ for all $q\in\mathbb{N}$;
\item
there exists a sequence $\{(\mu_{k_1,q},\ldots,\mu_{k_m,q})\}_{q\in\mathbb{N}}$ of points in $\mathbb{C}^m\setminus\{\mathbf{0}\}$ such that for all $q\in\mathbb{N}$ and all $1\leq i\leq n$:
\begin{equation*}
\sum_{j=1}^m \mu_{k_j,q}\frac{\partial f_{\mathbf{w}}^{k_j}}{\partial z_i}(\mathbf{p}_q)=\left\{
\begin{aligned}
& \lambda_q \, \bar p_{q,i} &&\mbox{if} && i\in I(\mathbf{w}),\\
& 0 &&\mbox{if} && i\notin I(\mathbf{w}),
\end{aligned}
\right.
\end{equation*}
where, for each $1\leq i\leq n$, $\bar p_{q,i}$ denotes the conjugate of the $i$th coordinate $p_{q,i}$  of~$\mathbf{p}_q$;
\item 
$\sum_{i\in I(\mathbf{w})} |p_{q,i}|^2\to 0$ as $q\to\infty$.
\end{enumerate}

For any $\zeta\in\mathbb{C}$ and any $\mathbf{z}\in \mathbb{C}^n$, let $\zeta*\mathbf{z}=((\zeta*\mathbf{z})_1,\ldots,(\zeta*\mathbf{z})_n)$ be the point of $\mathbb{C}^n$ defined by
\begin{equation*}
(\zeta*\mathbf{z})_i:= \zeta^{w_i}z_i=
\left\{
\begin{aligned}
&z_i &&\mbox{for} && i\in I(\mathbf{w}),\\
& \zeta^{w_i}z_i &&\mbox{for} && i\notin I(\mathbf{w}).\\
\end{aligned}
\right.
\end{equation*}
Then pick a sequence $\{\zeta_q\}_{q\in\mathbb{N}}$ of points in $\mathbb{C}^{*}$ that converges to zero sufficiently fast so that the sequence $\{\zeta_q*\mathbf{p}_q\}_{q\in\mathbb{N}}$ converges to the origin of $\mathbb{C}^{n}$. Clearly, $\{\zeta_q*\mathbf{p}_q\}_{q\in\mathbb{N}}$ also satisfies the above properties (1), (2) and (3). Indeed, for any $1\leq j\leq m$, we have
\begin{equation*}
f_{\mathbf{w}}^{k_j}(\zeta_q*\mathbf{p}_q)=\zeta_q^{d(\mathbf{w};f^{k_j})}f_{\mathbf{w}}^{k_j}(\mathbf{p}_q)=0,
\end{equation*}
 so $\{\zeta_q*\mathbf{p}_q\}_{q\in\mathbb{N}}$ satisfies (1). For each $1\leq i\leq n$, we also have
\begin{equation*}
\frac{\partial f_{\mathbf{w}}^{k_j}}{\partial z_i}(\zeta_q*\mathbf{p}_q)=
\zeta_q^{d(\mathbf{w};f^{k_j})-w_i} \, \frac{\partial f_{\mathbf{w}}^{k_j}}{\partial z_i}(\mathbf{p}_q),
\end{equation*}
and since $\zeta_q^{w_i}=1$ for all $i\in I(\mathbf{w})$ and $\zeta_q^{w_i}$ (which is non-zero) is independent of the index $j$ ($1\leq j\leq m$) for all $i\notin I(\mathbf{w})$, it follows that
\begin{equation*}
\sum_{j=1}^m \frac{\mu_{k_j,q}}{\zeta_q^{d(\mathbf{w};f^{k_j})}}
\frac{\partial f_{\mathbf{w}}^{k_j}}{\partial z_i}(\zeta_q*\mathbf{p}_q)=\left\{
\begin{aligned}
& \lambda_q \, \bar p_{q,i} &&\mbox{for} && i\in I(\mathbf{w}),\\
& 0 &&\mbox{for} && i\notin I(\mathbf{w}),
\end{aligned}
\right.
\end{equation*}
so that the sequence $\{\zeta_q*\mathbf{p}_q\}_{q\in\mathbb{N}}$ satisfies (2) with the complex numbers $\mu_{k_j,q} / \zeta_q^{d(\mathbf{w};f^{k_j})}$ ($1\leq j\leq m$). Finally,
\begin{equation*}
\sum_{i\in I(\mathbf{w})} |(\zeta_q*\mathbf{p}_q)_i|^2 = 
\sum_{i\in I(\mathbf{w})} |p_{q,i}|^2\to 0 
\end{equation*}
as $q\to\infty$, so $\{\zeta_q*\mathbf{p}_q\}_{q\in\mathbb{N}}$ also satisfies (3). Altogether, $\{\zeta_q*\mathbf{p}_q\}_{q\in\mathbb{N}}$ satisfies the properties (1)--(3).
Therefore, we can apply the Curve Selection Lemma to this situation in order to find a real analytic curve $\mathbf{a}(s)=(a_1(s),\ldots,a_n(s))$ in $\mathbb{C}^n$, $0\leq s\leq 1$,  and a family of complex numbers $\lambda(s)$, $0<s\leq 1$, such that:
\begin{enumerate}
\item[($1'$)]
$f_\mathbf{w}^{k_1}(\mathbf{a}(s))=\cdots=f_\mathbf{w}^{k_m}(\mathbf{a}(s))=0$ for all $s\not=0$;
\item[($2'$)]
there exists a real analytic curve $(\mu_{k_1}(s),\ldots,\mu_{k_m}(s))$ in $\mathbb{C}^m\setminus\{\mathbf{0}\}$, $0< s\leq 1$, such that for all $s\not=0$ and all $1\leq i\leq n$:
\begin{equation*}
\sum_{j=1}^m \mu_{k_j}(s)\, \frac{\partial f_{\mathbf{w}}^{k_j}}{\partial z_i}(\mathbf{a}(s))=\left\{
\begin{aligned}
& \lambda(s)\, \bar a_{i}(s) &&\mbox{if} &&i\in I(\mathbf{w}),\\
& 0 &&\mbox{if} &&i\notin I(\mathbf{w});
\end{aligned}
\right.
\end{equation*}
\item[($3'$)]
$\mathbf{a}(0)=\mathbf{0}$ and $\mathbf{a}(s)\in \mathbb{C}^{*n}$ for $s\not=0$.
\end{enumerate}

For each $1\leq i\leq n$, consider the Taylor expansion
\begin{equation*}
a_i(s)=b_i s^{v_i}+\cdots,
\end{equation*}
where $b_i\in\mathbb{C}^*$ and $v_i\in\mathbb{N}^*$.
Since the $v_i$'s are all positive, for each $1\leq j\leq m$ the face $\Delta\big(\mathbf{v};f_{\mathbf{w}}^{k_j}\big)$ is a compact face of $\Delta(\mathbf{w};f^{k_j})$, and hence $\Delta\big(\mathbf{v};f_{\mathbf{w}}^{k_j}\big)$ is a face of $\Gamma(f^{k_j})$, where $\mathbf{v}$ is the point of $\mathbb{N}^{*n}$ whose $i$th coordinate is $v_i$. Also, note that for each $j$, we have $d\big(\mathbf{v};f_{\mathbf{w}}^{k_j}\big)>0$, and since 
\begin{equation*}
0=f_{\mathbf{w}}^{k_j}(\mathbf{a}(s))=\big(f_{\mathbf{w}}^{k_j}\big)_{\mathbf{v}}(\mathbf{b})\, s^{d\big(\mathbf{v};f_{\mathbf{w}}^{k_j}\big)}+\cdots
\end{equation*}
 for all $s\not=0$, we also have $\big(f_{\mathbf{w}}^{k_j}\big)_{\mathbf{v}}(\mathbf{b})=0$, where $\mathbf{b}$ is the point of $\mathbb{C}^{*n}$ whose $i$th coordinate is $b_i$. (As usual, $\big(f_{\mathbf{w}}^{k_j}\big)_{\mathbf{v}}$ is the face function of $f_{\mathbf{w}}^{k_j}$ with respect to $\mathbf{v}$.)

Write $\mu_{k_j}(s)=\mu_{k_j} s^{g_j}+\cdots$, where $\mu_{k_j}\not=0$. If $\mu_{k_j}(s)\equiv 0$, then $g_j=\infty$. Let 
\begin{equation*}
\delta:=\mbox{min}\{d\big(\mathbf{v};f_{\mathbf{w}}^{k_1}\big)+g_1,\ldots,d\big(\mathbf{v};f_{\mathbf{w}}^{k_m}\big)+g_m\},
\end{equation*}
 and put
\begin{equation*}
\tilde\mu_{k_j} = 
\left\{
\begin{aligned}
& \mu_{k_j}  &&\mbox{if} && d\big(\mathbf{v};f_{\mathbf{w}}^{k_j}\big)+g_j=\delta,\\
& 0 &&\mbox{if} && d\big(\mathbf{v};f_{\mathbf{w}}^{k_j}\big)+g_j>\delta.
\end{aligned}
\right.
\end{equation*}

\begin{claim}\label{claim20210426}
There exists $i_0\in I(\mathbf{w})$ such that 
$
\sum_{j=1}^m \tilde\mu_{k_j} \frac{\partial \big(f_{\mathbf{w}}^{k_j}\big)_{\mathbf{v}}}{\partial z_{i_0}}(\mathbf{b})\not=0.
$
(We recall that $\mu_{k_j}(s)\not\equiv 0$ for at least an index $j$.)
\end{claim}

\begin{proof}
First, observe that for all $1\leq j\leq m$ and all $1\leq i\leq n$:
\begin{equation*}
\frac{\partial f_{\mathbf{w}}^{k_j}}{\partial z_{i}}(\mathbf{a}(s)) =
\frac{\partial \big(f_{\mathbf{w}}^{k_j}\big)_{\mathbf{v}}}{\partial z_{i}}(\mathbf{b})\, s^{d\big(\mathbf{v};f_{\mathbf{w}}^{k_j}\big)-v_i} + \cdots.
\end{equation*}
Thus, if the assertion in Claim \ref{claim20210426} fails, then the sum
\begin{equation*}
\sum_{j=1}^m \tilde\mu_{k_j} \frac{\partial \big(f_{\mathbf{w}}^{k_j}\big)_{\mathbf{v}}}{\partial z_{i}}(\mathbf{b})
\end{equation*}
vanishes for all $i\in I(\mathbf{w})$, and so, by ($2'$), it vanishes for all $1\leq i\leq n$. In other words, if $k_{j_1},\ldots,k_{j_p}$ are the elements of the set $\{k_1,\ldots,k_m\}$ for which $d\big(\mathbf{v};f_{\mathbf{w}}^{k_{j_\ell}}\big)+g_{j_\ell}=\delta$, $1\leq \ell\leq p$, then the vectors  
\begin{equation*}
\bigg(\frac{\partial \big(f_{\mathbf{w}}^{k_{j_1}}\big)_{\mathbf{v}}}{\partial z_{1}}(\mathbf{b}),\ldots,\frac{\partial \big(f_{\mathbf{w}}^{k_{j_1}}\big)_{\mathbf{v}}}{\partial z_{n}}(\mathbf{b})\bigg)
,\ldots,
\bigg(\frac{\partial \big(f_{\mathbf{w}}^{k_{j_p}}\big)_{\mathbf{v}}}{\partial z_{1}}(\mathbf{b}),\ldots,\frac{\partial \big(f_{\mathbf{w}}^{k_{j_p}}\big)_{\mathbf{v}}}{\partial z_{n}}(\mathbf{b})\bigg)
\end{equation*}
of $\mathbb{C}^n$ are linearly dependent, that is,
\begin{equation*}
d\big(f_{\mathbf{w}}^{k_{j_1}}\big)_{\mathbf{v}}(\mathbf{b})\wedge\cdots\wedge d\big(f_{\mathbf{w}}^{k_{j_p}}\big)_{\mathbf{v}}(\mathbf{b})=0.
\end{equation*}
As $\big(f_{\mathbf{w}}^{k_{j_\ell}}\big)_{\mathbf{v}} = f_{\mathbf{v}+\nu\mathbf{w}}^{k_{j_\ell}}$ for any sufficiently large integer $\nu\in\mathbb{N}$ (so that $\big(f_{\mathbf{w}}^{k_{j_\ell}}\big)_{\mathbf{v}}$ is the face function of $f^{k_{j_\ell}}$ with respect to the weight vector $\mathbf{v}+\nu\mathbf{w}$) and $\big(f_{\mathbf{w}}^{k_{j_\ell}}\big)_{\mathbf{v}}(\mathbf{b})=0$ for $1\leq \ell\leq p$, this contradicts the non-degeneracy of $V(f^{k_{j_1}},\ldots,f^{k_{j_p}})$ (see Assumptions \ref{ass-srfmf}). 
\end{proof}

Combined with ($2'$) again, Claim \ref{claim20210426} implies that $\lambda(s)$ is not constantly zero.
Write it as a Laurent series $\lambda(s)=\lambda_0 s^c+\cdots$, where $\lambda_0\not=0$. 
Then, still from ($2'$), we deduce that for all $1\leq i\leq n$:
\begin{equation}\label{r-20210607}
\sum_{j=1}^m \tilde\mu_{k_j} \frac{\partial \big(f_{\mathbf{w}}^{k_j}\big)_{\mathbf{v}}}{\partial z_i}(\mathbf{b})\, s^{\delta} + \cdots = 
\left\{
\begin{aligned}
& \lambda_0 \bar b_i\, s^{c+2v_i}+\cdots &&\mbox{if} &&i\in I(\mathbf{w}),\\
& 0 &&\mbox{if} &&i\notin I(\mathbf{w}).
\end{aligned}
\right.
\end{equation}
Put $S_i:=\sum_{j=1}^m \tilde\mu_{k_j} \frac{\partial \big(f_{\mathbf{w}}^{k_j}\big)_{\mathbf{v}}}{\partial z_i}(\mathbf{b})$, and define 
\begin{equation}\label{fvm}
v_0:=\mbox{min}\{v_i\, ;\, i\in I(\mathbf{w})\}
\quad\mbox{and}\quad
I_0 :=\{i\in I(\mathbf{w}) \, ;\,  v_i=v_0\}.
\end{equation}
Since the coefficient $\lambda_0\bar b_i$ on the right-hand side of \eqref{r-20210607} is non-zero and the set of indexes $i\in I(\mathbf{w})$ such that $S_i\not=0$ is not empty (see Claim \ref{claim20210426}), we have $\delta=c+2v_0$ and $S_i\not=0$ for any $i\in I_0$.
In fact, for any $1\leq i\leq n$, the following equality holds:
\begin{equation}\label{cIII}
S_i\equiv\sum_{j=1}^m \tilde\mu_{k_j} \frac{\partial \big(f_{\mathbf{w}}^{k_j}\big)_{\mathbf{v}}}{\partial z_i}(\mathbf{b})=
\left\{
\begin{aligned}
& \lambda_0 \bar b_i && \mbox{if} && i\in I_0,\\
& 0 && \mbox{if} && i\notin I_0.
\end{aligned}
\right.
\end{equation}
Since $I_0\not=\emptyset$ and $\big(f_{\mathbf{w}}^{k_j}\big)_{\mathbf{v}}(\mathbf{b})=0$ ($1\leq j\leq m$), combined with the Euler identity, the relation \eqref{cIII} implies
\begin{align*}
0 & = \sum_{j=1}^{m}\tilde\mu_{k_j}\cdot d\big(\mathbf{v};f_{\mathbf{w}}^{k_j}\big) \cdot \big(f_{\mathbf{w}}^{k_j}\big)_{\mathbf{v}}(\mathbf{b}) = \sum_{j=1}^{m}\tilde\mu_{k_j}\bigg(\sum_{i=1}^n v_i b_i \frac{\partial \big(f_{\mathbf{w}}^{k_j}\big)_{\mathbf{v}}}{\partial z_i}(\mathbf{b})\bigg) \\
& = \sum_{i\in I_0}v_i b_i\bigg( \sum_{j=1}^{m}\tilde\mu_{k_j} \frac{\partial \big(f_{\mathbf{w}}^{k_j}\big)_{\mathbf{v}}}{\partial z_i}(\mathbf{b}) \bigg)
=\lambda_0\cdot\sum_{i\in I_0} v_i |b_i|^2\not=0,
\end{align*}
which is a contradiction. This completes the proof of Lemma \ref{Lemma1}.
\end{proof}

\section{Uniformly stable family and uniform stable radius}\label{section-usfausr}

Now, let $f^1(t,\mathbf{z}),\ldots, f^{k_0}(t,\mathbf{z})$ be non-constant polynomial functions of $n+1$ complex variables $(t,\mathbf{z})=(t,z_1,\ldots,z_n)$ such that $f^k(t,\mathbf{0})=0$ for all $t\in\mathbb{C}$ and all $1\leq k\leq k_0$. As usual, for any $t\in\mathbb{C}$, we write $f^k_t(\mathbf{z}):=f^k(t,\mathbf{z})$. 

\begin{assumptions}\label{ass-sect-usr}
Throughout this section, we suppose that for any sufficiently small $t$ (say, $|t|\leq \tau_0$ for some $\tau_0>0$), the following two conditions hold true:
\begin{enumerate}
\item
for any $1\leq k\leq k_0$, the Newton boundary $\Gamma(f^k_t)$ is independent of $t$ (we may still have $\Gamma(f_t^{k})\not=\Gamma(f_t^{k'})$ for $k\not=k'$);
\item
for any $k_1,\ldots,k_m\in \{k_1,\ldots,k_0\}$, the germ at $\mathbf{0}$ of the variety $V(f^{k_1}_t,\ldots,f^{k_m}_t)$ is the germ of a non-degenerate complete intersection variety.
\end{enumerate}
\end{assumptions}

Note that (1) implies that the set $\mathcal{I}(f^k_t)$ is independent of $t$.

\subsection{Statements of the results of Section \ref{section-usfausr}}
By Lemma \ref{Lemma1}, we know that under Assumptions \ref{ass-sect-usr} there exists $\varepsilon_0>0$ such that for any $k_1,\ldots, k_m\in\{1,\ldots, k_0\}$, any $I\subseteq \{1,\ldots,n\}$ with $I\in\mathcal{I}(f_0^{k_1})\cap\cdots\cap\mathcal{I}(f_0^{k_m})$, any weight vector $\mathbf{w}\in\mathbb{N}^{I}$ and any $\lambda\in \mathbb{C}$,  if $\mathbf{a}\in\mathbb{C}^{I}$ satisfies the conditions (1) and (2) of this lemma for the functions $f_{0,\mathbf{w}}^{k_1,I},\ldots,f_{0,\mathbf{w}}^{k_m,I}$, then 
\begin{equation*}
\mathbf{a}\notin\bigg\{\mathbf{z}\in \mathbb{C}^{*I}\, ;\, \sum_{i\in I\cap I(\mathbf{w})}|z_i|^2\leq \varepsilon_0^2\bigg\}.
\end{equation*} 
(Here, $f_{0,\mathbf{w}}^{k,I}$ denotes the face function $(f_{0}^{k,I})_\mathbf{w}\equiv (f_{0}^{k,I})_{\Delta(\mathbf{w};f_{0}^{k,I})}$ of $f_{0}^{k,I}$ with respect to~$\mathbf{w}$.)

Once for all, let us fix such a number $\varepsilon_0$. Then we have the following result which asserts that if $t$ is small enough, then Lemma \ref{Lemma1} also holds for the functions $f_{t,\mathbf{w}}^{k_1,I},\ldots,f_{t,\mathbf{w}}^{k_m,I}$ with the same number $\varepsilon_0$.

\begin{lemma}\label{Lemma2}
Under Assumptions \ref{ass-sect-usr}, there exists $\tau$ with $0<\tau\leq \tau_0$ such that for any $t\in D_\tau:=\{t\in\mathbb{C}\, ;\, |t|\leq \tau\}$, any $k_1,\ldots, k_m\in\{1,\ldots, k_0\}$, any $I\subseteq \{1,\ldots,n\}$ with $I\in\mathcal{I}(f_t^{k_1})\cap\cdots\cap\mathcal{I}(f_t^{k_m})$, any weight vector $\mathbf{w}\in\mathbb{N}^{I}$ and any $\lambda\in \mathbb{C}$,  if $\mathbf{a}=(a_1,\ldots,a_n)\in\mathbb{C}^{I}$ satisfies the conditions (1) and (2) of Lemma \ref{Lemma1} for the functions $f_{t,\mathbf{w}}^{k_1,I},\ldots,f_{t,\mathbf{w}}^{k_m,I}$, that is, if:
\begin{enumerate}
\item
$f_{t,\mathbf{w}}^{k_1,I}(\mathbf{a})=\cdots=f_{t,\mathbf{w}}^{k_m,I}(\mathbf{a})=0$;
\item
there exists a $m$-tuple $(\mu_{k_1},\ldots,\mu_{k_m})\in\mathbb{C}^m\setminus\{\mathbf{0}\}$ such that for all $i\in I$:
\begin{equation*}
\sum_{j=1}^m \mu_{k_j}\frac{\partial f_{t,\mathbf{w}}^{k_j,I}}{\partial z_{i}}(\mathbf{a})=
\left\{
\begin{aligned}
& \lambda\bar a_{i} &&\mbox{if} && i\in I\cap I(\mathbf{w}),\\
& 0 &&\mbox{if} && i\in I\setminus I(\mathbf{w}),
\end{aligned}
\right.
\end{equation*}
where again $\bar a_i$ is the complex conjugate of $a_i$ and $I(\mathbf{w}):=\{i\in \{1,\ldots,n\};\, w_i=0\}$; 
\end{enumerate}
then we must have
\begin{equation*}
\mathbf{a}\notin\bigg\{\mathbf{z}\in \mathbb{C}^{*I}\, ;\, \sum_{i\in I\cap I(\mathbf{w})}|z_i|^2\leq \varepsilon_0^2\bigg\},
\end{equation*} 
where $\varepsilon_0$ is the number set above.
\end{lemma}

We shall prove Lemma \ref{Lemma2} in \S \ref{subsect-pol2}. Note that it generalizes Lemma 3 of \cite{O1} (obtained by taking $k_0=1$). Using Lemma \ref{Lemma2}, we shall prove the following second important theorem which recovers Theorem~2 of \cite{O1} (obtained for $k_0=1$).

Put $f(t,\mathbf{z}):=f^1(t,\mathbf{z})\cdots f^{k_0}(t,\mathbf{z})$, and as usual write $f_t(\mathbf{z}):=f(t,\mathbf{z})$. 

\begin{theorem}\label{smt}
Under Assumptions \ref{ass-sect-usr}, the family $\{f_t\}_{t\in D_\tau}$ is a uniformly stable family with uniform stable radius $\varepsilon_0$. (Here, $\tau$ is the number that appears in Lemma~\ref{Lemma2} and $\varepsilon_0$ is the number that we have fixed just before the statement of this lemma.)
\end{theorem}

We recall that the family $\{f_t\}_{t\in D_\tau}$ is said to be \emph{uniformly stable} with \emph{uniform stable radius} $\varepsilon_0$ if for any $0<\varepsilon_1\leq \varepsilon_2< \varepsilon_0$, there exists $\delta(\varepsilon_1,\varepsilon_2)>0$ such that for any $\eta\in\mathbb{C}$ with $0<\vert\eta\vert\leq\delta(\varepsilon_1,\varepsilon_2)$, the hypersurface $f_t^{-1}(\eta)$ is non-singular in $\mathring{B}_{\varepsilon_0}:=\{\mathbf{z}\in\mathbb{C}^n\, ;\, \Vert \mathbf{z} \Vert<\varepsilon_0\}$ and transversely intersects the sphere $S_{\varepsilon_{12}}:=\{\mathbf{z}\in\mathbb{C}^n\, ;\, \Vert \mathbf{z} \Vert=\varepsilon_{12}\}$ for any $\varepsilon_1\leq \varepsilon_{12}\leq \varepsilon_2$ and any $t\in D_\tau$.

We shall prove Theorem \ref{smt} in \S \ref{subsect-posmt}, but before giving the proof, let us state the first main theorem of this paper (Theorem \ref{fmtotp} below). For that purpose, we first observe that Theorem \ref{smt} has the following corollary which generalizes Corollary 1 of \cite{O1} (obtained by taking $k_0=1$). 

\begin{corollary}\label{cor1}
Under Assumptions \ref{ass-sect-usr}, the family $\{f_t\}_{t\in D_{\tau_0}}$ is a uniformly stable family.
\end{corollary}

\begin{proof}
By Lemma \ref{Lemma1}, for any $t_0\in D_{\tau_0}$, there exists $\varepsilon(t_0)>0$ such that for any $k_1,\ldots, k_m\in\{1,\ldots, k_0\}$, any $I\subseteq \{1,\ldots,n\}$ with $I\in\mathcal{I}(f_{t_0}^{k_1})\cap\cdots\cap\mathcal{I}(f_{t_0}^{k_m})$, any weight vector $\mathbf{w}\in\mathbb{N}^{I}$ and any $\lambda\in \mathbb{C}$,  if $\mathbf{a}\in\mathbb{C}^{I}$ satisfies the conditions (1) and (2) of this lemma for the functions $f_{t_0,\mathbf{w}}^{k_1,I},\ldots,f_{t_0,\mathbf{w}}^{k_m,I}$, then $\mathbf{a}$ does not belong to the set
\begin{equation}\label{ftheset}
\bigg\{\mathbf{z}\in \mathbb{C}^{*I}\, ;\, \sum_{i\in I\cap I(\mathbf{w})}|z_i|^2\leq \varepsilon(t_0)^2\bigg\}.
\end{equation} 
Then, by (the corresponding version of) Lemma \ref{Lemma2}, there exists $\tau(t_0)>0$ such that for any $t\in D_{\tau(t_0)}(t_0):=\{t\in\mathbb{C}\, ;\, |t-t_0|\leq \tau(t_0)\}$, any $k_1,\ldots, k_m\in\{1,\ldots, k_0\}$, any $I\subseteq \{1,\ldots,n\}$ with $I\in\mathcal{I}(f_t^{k_1})\cap\cdots\cap\mathcal{I}(f_t^{k_m})$, any weight vector $\mathbf{w}\in\mathbb{N}^{I}$ and any $\lambda\in \mathbb{C}$,  if $\mathbf{a}\in\mathbb{C}^{I}$ satisfies the conditions (1) and (2) of Lemma \ref{Lemma1} for the functions $f_{t,\mathbf{w}}^{k_1,I},\ldots,f_{t,\mathbf{w}}^{k_m,I}$, then $\mathbf{a}$ does not belong to the set \eqref{ftheset}.
Now, applying (the corresponding version of) Theorem \ref{smt} shows that the family $\{f_t\}_{t\in D_{\tau(t_0)}}$ is a uniformly stable family with uniform stable radius $\varepsilon(t_0)$. Corollary \ref{cor1} then follows from the compactness of the disc $D_{\tau_0}$.
\end{proof}

Now, by \cite[Lemma 2]{O1}, we know that if $\{f_t\}_{t\in D_{\tau_0}}$ is a uniformly stable family\textemdash say, with uniform stable radius $\varepsilon$\textemdash then the Milnor fibrations of $f_t$ and $f_0$ at $\mathbf{0}$ are \emph{isomorphic} for all  $t\in D_{\tau_0}$, that is, for all such $t$'s there exists a fibre-preserving diffeomorphism 
\begin{equation*}
\mathring{B}_\varepsilon\cap f_t^{-1}\Big(S_{\delta(\varepsilon,\frac{\varepsilon}{2})}\Big)
\overset{\sim}{\longrightarrow}
\mathring{B}_\varepsilon\cap f_0^{-1}\Big(S_{\delta(\varepsilon,\frac{\varepsilon}{2})}\Big),
\end{equation*}
where $\delta(\varepsilon,\frac{\varepsilon}{2})$ is the number that appears in the definition of a ``uniform stable family'' given just after the statement of Theorem \ref{smt}, and where $S_{\delta(\varepsilon,\frac{\varepsilon}{2})}:=\{z\in\mathbb{C}\, ;\, \vert z\vert=\delta(\varepsilon,\frac{\varepsilon}{2})\}$.
Combining this result with Corollary \ref{cor1} gives our first main theorem, the statement of which is as follows. Again, the special case $k_0=1$ (for which the functions $f_t$ are necessarily non-degenerate) is already contained in \cite{O1}.

\begin{theorem}\label{fmtotp}
Under Assumptions \ref{ass-sect-usr}, the Milnor fibrations of $f_t$ and $f_0$ at $\mathbf{0}$ are isomorphic for all $t\in D_{\tau_0}$.
\end{theorem}

The following two subsections (\S \ref{subsect-posmt} and \S\ref{subsect-pol2}) are devoted to the proofs of Theorem \ref{smt} and Lemma \ref{Lemma2} respectively.

\subsection{Proof of Theorem \ref{smt}}\label{subsect-posmt}
It is along the same lines as the proof of Theorem \ref{fmt}. We start with the following claim which plays a role similar to that of \cite[Corollary 2.8]{Milnor} in the proof of Theorem \ref{fmt}.

\begin{claim}\label{lafpsmt}
There exists $\delta>0$ such that for any $\eta\in\mathbb{C}$ with $0<|\eta|\leq\delta$, the hypersurface $f_t^{-1}(\eta)$ is non-singular in $\mathring{B}_{\varepsilon_0}$ for any $t\in D_{\tau}$. (Of course, we work under Assumptions \ref{ass-sect-usr}.)
\end{claim}

We postpone the proof of this claim at the end of \S \ref{subsect-posmt}, and we first complete the proof of Theorem \ref{smt}.
We argue by contradiction. By Claim \ref{lafpsmt}, if the assertion in Theorem \ref{smt} is false, then it follows from the Curve Selection Lemma that there exists a real analytic curve $(t(s),\mathbf{z}(s))=(t(s),z_1(s),\ldots,z_n(s))$ in $D_\tau\times \mathring{B}_{\varepsilon_0}$, $0\leq s\leq 1$, and a family of complex numbers $\lambda(s)$, $0< s\leq 1$, such that the following three condition holds:
\begin{enumerate}
\item[(i)]
$\frac{\partial f_{t(s)}}{\partial z_i}(\mathbf{z}(s)) = \lambda(s) \bar z_i(s)$ for $1\leq i\leq n$ and $s\not=0$;
\item[(ii)]
$f_{t(0)}(\mathbf{z}(0))=0$ but $f_{t(s)}(\mathbf{z}(s))\not=0$ for $s\not=0$;
\item[(iii)]
there exists $\varepsilon>0$ such that $\varepsilon\leq \Vert \mathbf{z}(s)\Vert\leq \varepsilon_0$.
\end{enumerate}
By (i) and (ii), $\lambda(s)\not\equiv 0$, and we can express it as a Laurent series
\begin{equation*}
\lambda(s)=\lambda_0 s^c+\cdots,
\end{equation*}
where $\lambda_0\in\mathbb{C}^*$.
Let $I:=\{i\, ;\, z_i(s)\not\equiv 0\}$. 
By (ii), $I\in\mathcal{I}(f_{t(s)})$, and hence $I\in\mathcal{I}(f^{1}_{t(s)})\cap\cdots\cap\mathcal{I}(f^{k_0}_{t(s)})$. For each $i\in I$, consider the Taylor expansion
\begin{equation*}
z_i(s)=a_i s^{w_i}+\cdots,
\end{equation*}
where $a_i\in\mathbb{C}^*$ and $w_i\in\mathbb{N}$. 
The following is the counterpart of Claim \ref{claim1}.

\begin{claim}\label{claim1f2}
There exists $1\leq k\leq k_0$ such that $f^{k,I}_{t(0),\mathbf{w}}(\mathbf{a})=0$, where again $\mathbf{a}$ and $\mathbf{w}$ are the points in $\mathbb{C}^{*I}$ and $\mathbb{N}^{I}$, respectively, whose $i$th coordinates ($i\in I$) are $a_{i}$ and $w_{i}$ respectively.
\end{claim}

Again, we shall prove this claim later. First, we complete the proof of the theorem. Once more, hereafter, to simplify the notation, we shall assume that $I=\{1,\ldots,n\}$, so that the function $f_t^{k,I}$ is simply written as $f_t^k$, the intersection $I\cap I(\mathbf{w})$ is written as $I(\mathbf{w})$ (where, as in Lemma \ref{Lemma2}, $I(\mathbf{w})$ is the set of all indexes $i\in \{1,\ldots,n\}$ for which $w_i=0$), and so on. 

Look at the set consisting of all integers $k$ for which $f^{k}_{t(0),\mathbf{w}}(\mathbf{a})=0$. By Claim \ref{claim1f2}, this set is not empty. As in the proof of Theorem \ref{fmt}, we assume that $f^{k}_{t(0),\mathbf{w}}(\mathbf{a})$ vanishes for $1\leq k\leq k_0'$ and does not vanish for $k'_0+1\leq k\leq k_0$, and we write $f=f^1\cdots f^{k_0'}\cdot h$ where $h:=f^{k_0'+1}\cdots f^{k_0}$ if $k_0'\leq k_0-1$ and $h:=1$ if $k_0'=k_0$; finally, for each $1\leq k\leq k'_0$, we put
\begin{equation*}
e_k:=d\big(\mathbf{w};f_{t(0)}^{k}\big)-\mbox{ord}\, f_{t(s)}^{k}(\mathbf{z}(s))+\sum_{\ell=1}^{k_0'} \mbox{ord}\, f_{t(s)}^{\ell}(\mathbf{z}(s))
\end{equation*}
(where, as usual, $\mbox{ord} f_{t(s)}^{\ell}(\mathbf{z}(s))$ means the order, in $s$, of the expression $f_{t(s)}^{\ell}(\mathbf{z}(s))\equiv f^{\ell}(t(s),\mathbf{z}(s))$), 
and we suppose that 
\begin{equation*}
e_{\mbox{\tiny min}}:=e_1=\cdots=e_{k_0''}<e_{k_0''+1}\leq \cdots\leq e_{k_0'}.
\end{equation*}
Note that the equality $\Gamma_+\big(f^{k}_{t(s)}\big)=\Gamma_+\big(f^{k}_{t(0)}\big)$ implies $\Delta\big(\mathbf{w};f^{k}_{t(s)}\big)=\Delta\big(\mathbf{w};f^{k}_{t(0)}\big)$ and $d\big(\mathbf{w};f^{k}_{t(s)}\big)=d\big(\mathbf{w};f^{k}_{t(0)}\big)=d(\hat{\mathbf{w}};f^k)$ for all $s$, where $\hat{\mathbf{w}}=(w_0,\mathbf{w})$ with $w_0$ defined by the Taylor expansion $t(s):=t_0s^{w_0}+\cdots$, $t_0\not=0$.
Still as in the proof of Theorem \ref{fmt} (see \eqref{e1}--\eqref{ibciol}), it follows from the relation (i) that there exist non-zero complex numbers $\mu_1,\ldots,\mu_{k''_0}$ such that for any $1\leq i\leq n$:
\begin{equation*}
\sum_{k=1}^{k''_0}\frac{\partial {f^{k}_{t(0),\mathbf{w}}}}{\partial z_{i}} (\mathbf{a})\cdot \mu_k\cdot s^{d(\hat{\mathbf{w}};h)+e_{\mbox{\tiny min}}}+\cdots = \lambda_0 \bar a_{i}s^{c+2w_{i}}+\cdots,
\end{equation*}
and since $\lambda_0 \bar a_{i}\not=0$ and $I(\mathbf{w})\not=\emptyset$ (by (iii)), by the same argument as the one given after \eqref{ibciol}, we deduce that the sum 
\begin{equation*}
S_i:=\sum_{k=1}^{k''_0} \mu_k 
\frac{\partial {f^{k}_{t(0),\mathbf{w}}}}{\partial z_{i}} (\mathbf{a})
\end{equation*}
vanishes for all $i\notin I(\mathbf{w})$. If it also vanishes for all $i\in I(\mathbf{w})$, then we get a contradiction with Lemma \ref{Lemma2} because $\mathbf{z}(s)\in\mathring{B}_{\varepsilon_0}$, and hence,
\begin{equation}\label{pthm45-di}
\sum_{i\in I(\mathbf{w})}\vert a_i \vert^2=
\Vert \mathbf{z}(0)\Vert^2\leq \varepsilon_0^2.
\end{equation} 
If there is an index $i_0\in I(\mathbf{w})$ such that $S_{i_0}\not=0$, then 
\begin{equation*}
S_i=\sum_{k=1}^{k''_0} \mu_k 
\frac{\partial {f^{k}_{t(0),\mathbf{w}}}}{\partial z_{i}} (\mathbf{a})
=\left\{
\begin{aligned}
& \lambda_0 \bar a_i &&\mbox{for}&& i\in I(\mathbf{w}),\\
& 0 &&\mbox{for}&& i\notin I(\mathbf{w}),
\end{aligned}
\right.
\end{equation*}
and still by \eqref{pthm45-di}, we get a new contradiction with Lemma \ref{Lemma2}.

To complete the proof of Theorem \ref{smt}, it remains to prove Claims \ref{lafpsmt} and \ref{claim1f2}. We start with the proof of  Claim \ref{claim1f2}.

\begin{proof}[Proof of Claim \ref{claim1f2}]
It is similar to the proof of Claim \ref{claim1}. Again, we assume $I=\{1,\ldots,n\}$, so that $f^{k,I}_{t(0),\mathbf{w}}=f^{k}_{t(0),\mathbf{w}}$. We argue by contradiction. Suppose that $f^{k}_{t(0),\mathbf{w}}(\mathbf{a})\not=0$ for all $1\leq k\leq k_0$. Then $f^k_{\hat{\mathbf{w}}}(t_0,\mathbf{a})=f^k_{t(0),\mathbf{w}}(\mathbf{a})\not=0$ and $d\big(\mathbf{w};f_{t(0)}^{k}\big)=d(\hat{\mathbf{w}};f^k)=\mbox{ord}\, f_{t(s)}^{k}(\mathbf{z}(s))$ for all $1\leq k\leq k_0$ (where $\hat{\mathbf{w}}$ and $t_0$ are defined as above), and by (ii), there exists $1\leq k_1\leq k_0$ such that $f_{t(0)}^{k_1}(\mathbf{z}(0))=0$. If $I(\mathbf{w})=\{1,\ldots,n\}$, then $d\big(\mathbf{w};f_{t(0)}^{k_1}\big)=0$ and
\begin{equation*}
f_{t(s)}^{k_1}(\mathbf{z}(s))=f^{k_1}_{t(0),\mathbf{w}}(\mathbf{a})\, s^0+\cdots, 
\end{equation*}
so that $0=f_{t(0)}^{k_1}(\mathbf{z}(0))=f^{k_1}_{t(0),\mathbf{w}}(\mathbf{a})$, which is a contradiction. If $I(\mathbf{w})$ is a proper subset of $\{1,\ldots,n\}$ and $d\big(\mathbf{w};f_{t(0)}^{k_1}\big)\not=0$, then, exactly as in the proof of Claim \ref{claim1}, if $e:=\sum_{k=1}^{k_0} \mbox{ord}\, f_{t(s)}^{k}(\mathbf{z}(s))$, then for any $1\leq i\leq n$:
\begin{equation}\label{egalite982021}
\sum_{k=1}^{k_0} \bigg(\prod_{\genfrac{}{}{0pt}{}{1\leq \ell\leq k_0}{\ell\not=k}}f^{\ell}_{t(0),\mathbf{w}}(\mathbf{a})\bigg)\cdot\frac{\partial {f^{k}_{t(0),\mathbf{w}}}}{\partial z_{i}} (\mathbf{a})\cdot s^{e}+\cdots = \lambda_0 \bar a_{i}s^{c+2w_{i}}+\cdots.
\end{equation}
As above, since $\lambda_0 \bar a_{i}\not=0$ and $I(\mathbf{w})\not=\emptyset$, this implies that the sum 
\begin{equation*}
\sum_{k=1}^{k_0} \bigg(\prod_{\genfrac{}{}{0pt}{}{1\leq \ell\leq k_0}{\ell\not=k}}f^{\ell}_{t(0),\mathbf{w}}(\mathbf{a})\bigg)\cdot\frac{\partial {f^{k}_{t(0),\mathbf{w}}}}{\partial z_{i}} (\mathbf{a})
\end{equation*}
vanishes for all $i\notin I(\mathbf{w})$, and using the Euler identity, we get exactly the same contradiction as in the proof of Claim \ref{claim1}.
\end{proof}

Now we prove Claim \ref{lafpsmt}.

\begin{proof}[Proof of Claim \ref{lafpsmt}]
The argument is very similar to that given in the proof of Theorem \ref{smt}.
 We argue by contradiction. If the assertion in the claim is false, then, by the Curve Selection Lemma, there exists a real analytic curve $(t(s),\mathbf{z}(s))=(t(s),z_1(s),\ldots,z_n(s))$ in $D_\tau\times \mathring{B}_{\varepsilon_0}$, $0\leq s\leq 1$,  such that the following two condition holds:
\begin{enumerate}
\item[(i)]
$\frac{\partial f_{t(s)}}{\partial z_i}(\mathbf{z}(s)) = 0$ for $1\leq i\leq n$;
\item[(ii)]
$f_{t(0)}(\mathbf{z}(0))=0$ but $f_{t(s)}(\mathbf{z}(s))\not=0$ for $s\not=0$;
\end{enumerate}
Let $I:=\{i\, ;\, z_i(s)\not\equiv 0\}$. 
By (ii), $I\in\mathcal{I}(f^{1}_{t(s)})\cap\cdots\cap\mathcal{I}(f^{k_0}_{t(s)})$. For each $i\in I$, consider the Taylor expansion
\begin{equation*}
z_i(s)=a_i s^{w_i}+\cdots,
\end{equation*}
where $a_i\in\mathbb{C}^*$ and $w_i\in\mathbb{N}$. 

\begin{claim}\label{claim1f}
There exists $1\leq k\leq k_0$ such that $f^{k,I}_{t(0),\mathbf{w}}(\mathbf{a})=0$, where again $\mathbf{a}$ and $\mathbf{w}$ are the points in $\mathbb{C}^{*I}$ and $\mathbb{N}^{I}$, respectively, whose $i$th coordinates ($i\in I$) are $a_{i}$ and $w_{i}$ respectively.
\end{claim}

The proof of Claim \ref{claim1f} is completely similar to that of Claim \ref{claim1f2}. The only difference is that the right-hand side of the equality \eqref{egalite982021} is now zero. However this does not change anything in the argument.

Once more, we assume $I=\{1,\ldots,n\}$, so that $f_{t(0),\mathbf{w}}^{k,I}=f_{t(0),\mathbf{w}}^k$, and we
look at the set consisting of all integers $k$ for which $f^{k}_{t(0),\mathbf{w}}(\mathbf{a})=0$. By Claim \ref{claim1f}, this set is not empty. As in the proofs of Theorems \ref{fmt} or \ref{smt}, we assume that $f^{k}_{t(0),\mathbf{w}}(\mathbf{a})$ vanishes for $1\leq k\leq k_0'$ and does not vanish for $k'_0+1\leq k\leq k_0$, and we write $f=f^1\cdots f^{k_0'}\cdot h$ where $h:=f^{k_0'+1}\cdots f^{k_0}$ if $k_0'\leq k_0-1$ and $h:=1$ if $k_0'=k_0$; finally, for each $1\leq k\leq k'_0$, we put
\begin{equation*}
e_k:=d\big(\mathbf{w};f_{t(0)}^{k}\big)-\mbox{ord}\, f_{t(s)}^{k}(\mathbf{z}(s))+\sum_{\ell=1}^{k_0'} \mbox{ord}\, f_{t(s)}^{\ell}(\mathbf{z}(s)),
\end{equation*} 
and we suppose that 
\begin{equation*}
e_{\mbox{\tiny min}}:=e_1=\cdots=e_{k_0''}<e_{k_0''+1}\leq \cdots\leq e_{k_0'}.
\end{equation*}
Still as in the proofs of Theorems \ref{fmt} or \ref{smt}, it follows from the relation (i) that there exist non-zero complex numbers $\mu_1,\ldots,\mu_{k''_0}$ such that for any $1\leq i\leq n$:
\begin{equation*}
\sum_{k=1}^{k''_0}\frac{\partial {f^{k}_{t(0),\mathbf{w}}}}{\partial z_{i}} (\mathbf{a})\cdot \mu_k\cdot s^{d(\hat{\mathbf{w}};h)+e_{\mbox{\tiny min}}}+\cdots = 0,
\end{equation*}
and hence  
$
\sum_{k=1}^{k''_0} \mu_k 
\frac{\partial {f^{k}_{t(0),\mathbf{w}}}}{\partial z_{i}} (\mathbf{a})=0.
$
In other words, the vectors 
\begin{equation*}
\bigg(\frac{\partial {f^{1}_{t(0),\mathbf{w}}}}{\partial z_{1}} (\mathbf{a}),\ldots,
\frac{\partial {f^{1}_{t(0),\mathbf{w}}}}{\partial z_{n}} (\mathbf{a})\bigg)
,\ldots, 
\bigg(\frac{\partial {f^{k_0''}_{t(0),\mathbf{w}}}}{\partial z_{1}} (\mathbf{a}),\ldots,
\frac{\partial {f^{k_0''}_{t(0),\mathbf{w}}}}{\partial z_{n}} (\mathbf{a})\bigg)
\end{equation*}
of $\mathbb{C}^n$ are linearly dependent, that is, 
\begin{equation*}
df^{1}_{t(0),\mathbf{w}}(\mathbf{a})\wedge\cdots\wedge
df^{k_0''}_{t(0),\mathbf{w}}(\mathbf{a})=0,
\end{equation*}
which contradicts the non-degeneracy of $V\big(f^{1}_{t(0)},\ldots,f^{k_0''}_{t(0)}\big)$ if $I(\mathbf{w})=\emptyset$. In the case where $I(\mathbf{w})\not=\emptyset$, we cannot proceed like that. However, in this case, Lemma \ref{Lemma2} (applied with $\lambda=0$) implies 
\begin{equation*}
\mathbf{a}\notin\bigg\{\mathbf{z}\in \mathbb{C}^{*n}\, ;\, \sum_{i\in I(\mathbf{w})}|z_i|^2\leq \varepsilon_0^2\bigg\},
\end{equation*} 
and since $\mathbf{z}(s)\in\mathring{B}_{\varepsilon_0}$, we also have
\begin{equation*}
\sum_{i\in I(\mathbf{w})}|a_i|^2\leq\sum_{i=1}^n |z_i(0)|^2=\Vert\mathbf{z}(0)\Vert^2<\varepsilon_0^2,
\end{equation*}
which is a contradiction. 
\end{proof}

\subsection{Proof of Lemma \ref{Lemma2}}\label{subsect-pol2}
If the assertion of this lemma fails for some $k_1,\dots,k_m$, $I$ and $\Delta(\mathbf{w};f_0^{k_1}),\ldots,\Delta(\mathbf{w};f_0^{k_m})$ such that $I\cap I(\mathbf{w})=\emptyset$, then, as in the proof of Lemma \ref{Lemma1}, we get a contradiction with the non-degeneracy condition (see Assumptions \ref{ass-sect-usr} and Remark \ref{rem-rI}). 

Now, assume that the assertion fails for some $k_1,\dots,k_m$, $I$ and $\Delta(\mathbf{w};f_0^{k_1}),\ldots,\Delta(\mathbf{w};f_0^{k_m})$ such that $I\cap I(\mathbf{w})\not=\emptyset$. 
Again, without loss of generality, and in order to simplify the notation, we assume that $I=\{1,\ldots,n\}$, so that $f^{k,I}_{t,\mathbf{w}}=f^{k}_{t,\mathbf{w}}$, $I\cap I(\mathbf{w})=I(\mathbf{w})$, $\mathbb{C}^{*I}=\mathbb{C}^{*n}$, etc.
Then there exist sequences $\{\mathbf{p}_q\}_{q\in\mathbb{N}}$, $\{\lambda_q\}_{q\in\mathbb{N}}$ and $\{t_q\}_{q\in\mathbb{N}}$ of points in $\mathbb{C}^{*n}$, $\mathbb{C}$ and $\mathbb{C}^*$, respectively, such that:
\begin{enumerate}
\item
$f_{t_q,\mathbf{w}}^{k_1}(\mathbf{p}_q)=\cdots=f_{t_q,\mathbf{w}}^{k_m}(\mathbf{p}_q)=0$ for all $q\in\mathbb{N}$;
\item
there exists a sequence $\{(\mu_{k_1,q},\ldots,\mu_{k_m,q})\}_{q\in\mathbb{N}}$ of points in $\mathbb{C}^m\setminus\{\mathbf{0}\}$ such that for all $q\in\mathbb{N}$ and all $1\leq i\leq n$:
\begin{equation*}
\sum_{j=1}^m \mu_{k_j,q}\, \frac{\partial f_{t_q,\mathbf{w}}^{k_j}}{\partial z_i}(\mathbf{p}_q)=\left\{
\begin{aligned}
& \lambda_q \, \bar p_{q,i} &&\mbox{if} && i\in I(\mathbf{w}),\\
& 0 &&\mbox{if} && i\notin I(\mathbf{w}),
\end{aligned}
\right.
\end{equation*}
where, for each $1\leq i\leq n$, $\bar p_{q,i}$ denotes the conjugate of the $i$th coordinate $p_{q,i}$  of~$\mathbf{p}_q$;
\item 
$\sum_{i\in I(\mathbf{w})} |p_{q,i}|^2\leq \varepsilon_0^2$ and $t_q\to 0$ as $q\to\infty$.
\end{enumerate}
(Again, $f_{t_q,\mathbf{w}}^{k_j}$ denotes the face function $(f_{t_q}^{k_j})_\mathbf{w}\equiv (f_{t_q}^{k_j})_{\Delta\big(\mathbf{w};f_{t_q}^{k_j}\big)}$ of $f_{t_q}^{k_j}$ with respect to $\mathbf{w}$.)
By an argument similar to that used in the proof of Lemma \ref{Lemma1}, we can assume that the sequences $\{p_{q,i}\}_{q\in\mathbb{N}}$ converge to $0$ for all $i\notin I(\mathbf{w})$, so that, once again, we can apply the Curve Selection Lemma to get a real analytic curve 
$
(t(s),\mathbf{a}(s))=(t(s),a_1(s),\ldots,a_n(s))
$
in $\mathbb{C}\times\mathbb{C}^{n}$, $0\leq s\leq 1$, and a family of complex numbers $\lambda(s)$, $0<s\leq 1$, such that:
\begin{enumerate}
\item[($1'$)]
$f_{t(s),\mathbf{w}}^{k_1}(\mathbf{a}(s))=\cdots=f_{t(s),\mathbf{w}}^{k_m}(\mathbf{a}(s))=0$ for all $s\not=0$;
\item[($2'$)]
there exists a real analytic curve $(\mu_{k_1}(s),\ldots,\mu_{k_m}(s))$ in $\mathbb{C}^m\setminus\{\mathbf{0}\}$, $0< s\leq 1$, such that for all $s\not=0$ and all $1\leq i\leq n$:
\begin{equation*}
\sum_{j=1}^m \mu_{k_j}(s)\, \frac{\partial f_{t(s),\mathbf{w}}^{k_j}}{\partial z_i}(\mathbf{a}(s))=\left\{
\begin{aligned}
& \lambda(s)\, \bar a_{i}(s) &&\mbox{if} &&i\in I(\mathbf{w}),\\
& 0 &&\mbox{if} &&i\notin I(\mathbf{w});
\end{aligned}
\right.
\end{equation*}
\item[($3'$)]
$\sum_{i\in I(\mathbf{w})} |a_i(s)|^2\leq \varepsilon_0^2$, $t(0)=0$, $a_i(0)=0$ for $i\notin I(\mathbf{w})$, and $\mathbf{a}(s)\in \mathbb{C}^{*n}$ for $s\not=0$.
\end{enumerate}

For each $1\leq i\leq n$, consider the Taylor expansion
\begin{equation*}
a_i(s)=b_i s^{v_i}+\cdots,
\end{equation*}
where $b_i\in\mathbb{C}^*$ and $v_i\in\mathbb{N}$, and put $v_{\mbox{\tiny min}}:=\mbox{min}\{v_1,\ldots,v_n\}$. Then we divide the proof into two cases depending on whether $v_{\mbox{\tiny min}}=0$ or $v_{\mbox{\tiny min}}>0$.
Let us first assume $v_{\mbox{\tiny min}}>0$. In this case, the proof is similar to that of Lemma \ref{Lemma1}. Indeed, exactly as in this proof, for each $1\leq j\leq m$ the face $\Delta\big(\mathbf{v};f_{0,\mathbf{w}}^{k_j}\big)$ is a (compact) face of $\Gamma\big(f_{0}^{k_j}\big)$ and $d\big(\mathbf{v};f_{0,\mathbf{w}}^{k_j}\big)>0$. Since $\Gamma_+\big(f^{k_j}_t\big)$\textemdash and hence $\Delta\big(\mathbf{w};f^{k_j}_t\big)$\textemdash is independent of $t$, we have
\begin{equation*}
0=f_{t(s),\mathbf{w}}^{k_j}(\mathbf{a}(s))=f_{0,\mathbf{w},\mathbf{v}}^{k_j}(\mathbf{b})\cdot s^{d\big(\mathbf{v};f_{0,\mathbf{w}}^{k_j}\big)} + \cdots
\end{equation*}
for all $s\not=0$,
and hence $f_{0,\mathbf{w},\mathbf{v}}^{k_j}(\mathbf{b})=0$, where $\mathbf{v}$ and $\mathbf{b}$ are the points of $\mathbb{N}^{*n}$ and $\mathbb{C}^{*n}$, respectively, whose $i$th coordinates are $v_i$ and $b_i$ respectively. Here, according to our notation, by $f_{0,\mathbf{w},\mathbf{v}}^{k_j}$ we mean the face function $\big(\big(f_{0}^{k_j}\big)_{\mathbf{w}}\big)_{\mathbf{v}}$ of $f_{0,\mathbf{w}}^{k_j}\equiv\big(f_{0}^{k_j}\big)_{\mathbf{w}}$ with respect to $\mathbf{v}$.

Write $\mu_{k_j}(s)=\mu_{k_j} s^{g_j}+\cdots$, where $\mu_{k_j}\not=0$. Again, if $\mu_{k_j}(s)\equiv 0$, then $g_j=\infty$. Put 
\begin{align*}
\delta:=\mbox{min}\big\{d\big(\mathbf{v};f_{0,\mathbf{w}}^{k_1}\big)+g_1,\ldots,d\big(\mathbf{v};f_{0,\mathbf{w}}^{k_m}\big)+g_m\big\},
\end{align*}
 and define $\tilde\mu_{k_j}$ to be equal to $\mu_{k_j}$ or $0$ depending on whether $d\big(\mathbf{v};f_{0,\mathbf{w}}^{k_j}\big)+g_j$ is equal to $\delta$ or not respectively.

\begin{claim}\label{claim2} 
There exists $i_0\in I(\mathbf{w})$ such that 
$
\sum_{j=1}^m \tilde\mu_{k_j} \frac{\partial f_{0,\mathbf{w},\mathbf{v}}^{k_j}}{\partial z_{i_0}}(\mathbf{b})\not=0.
$
\end{claim}

\begin{proof}
It is along the same lines as the proof of Claim \ref{claim20210426}.
More precisely, since $\Gamma_+\big(f^{k_j}_t\big)$ is independent of $t$, we have
\begin{equation*}
\frac{\partial f_{t(s),\mathbf{w}}^{k_j}}{\partial z_i}(\mathbf{a}(s)) =
\frac{\partial f_{0,\mathbf{w},\mathbf{v}}^{k_j}}{\partial z_i}(\mathbf{b})\, s^{d\big(\mathbf{v};f_{0,\mathbf{w}}^{k_j}\big)-v_i}+\cdots
\end{equation*}
for all $1\leq j\leq m$ and all $1\leq i\leq n$.
Thus, if the assertion in Claim \ref{claim2} fails, then the sum 
\begin{equation*}
\sum_{j=1}^m \tilde\mu_{k_j} \frac{\partial f_{0,\mathbf{w},\mathbf{v}}^{k_j}}{\partial z_{i}}(\mathbf{b})
\end{equation*}
vanishes for all $i\in I(\mathbf{w})$, and so, by ($2'$), it vanishes for all $1\leq i\leq n$. As in the proof of Claim \ref{claim20210426}, this implies that
\begin{equation*}
df_{0,\mathbf{w},\mathbf{v}}^{k_{j_1}}(\mathbf{b})\wedge\cdots\wedge df_{0,\mathbf{w},\mathbf{v}}^{k_{j_p}}(\mathbf{b})=0,
\end{equation*}
where the $k_{j_\ell}$'s ($1\leq \ell\leq p$)
are the elements of $\{k_1,\ldots,k_m\}$ for which $d\big(\mathbf{v};f_{0,\mathbf{w}}^{k_{j_\ell}}\big)+g_{j_\ell}=\delta$.
Since $f_{0,\mathbf{w},\mathbf{v}}^{k_{j_\ell}} = f_{0,\mathbf{v}+\nu\mathbf{w}}^{k_{j_\ell}}$ for any sufficiently large integer $\nu\in\mathbb{N}$ (so that $f_{0,\mathbf{w},\mathbf{v}}^{k_{j_\ell}}$ is the face function of $f_{0}^{k_{j_\ell}}$ with respect to the weight vector $\mathbf{v}+\nu\mathbf{w}$) and $f_{0,\mathbf{w},\mathbf{v}}^{k_{j_\ell}}(\mathbf{b})=0$ for $1\leq \ell \leq p$, and since $v_i+\nu w_i>0$ for all $1\leq i\leq n$, this contradicts the non-degeneracy of $V\big(f_0^{k_{j_1}},\ldots,f_0^{k_{j_p}}\big)$ (see Assumptions \ref{ass-sect-usr}).
\end{proof}

Combined with ($2'$) again, Claim \ref{claim2} implies that $\lambda(s)$ is not constantly zero.
Write it as a Laurent series $\lambda(s)=\lambda_0 s^c+\cdots$, where $\lambda_0\not=0$. Then, still from ($2'$), we deduce that for all $1\leq i\leq n$:
\begin{equation*}
\sum_{j=1}^m \tilde\mu_{k_j} \frac{\partial f_{0,\mathbf{w},\mathbf{v}}^{k_j}}{\partial z_i}(\mathbf{b})\, s^{\delta} + \cdots = 
\left\{
\begin{aligned}
& \lambda_0 \bar b_i\, s^{c+2v_i}+\cdots &&\mbox{if} &&i\in I(\mathbf{w}),\\
& 0 &&\mbox{if} &&i\notin I(\mathbf{w}).
\end{aligned}
\right.
\end{equation*}
Now, put $S_i:=\sum_{j=1}^m \tilde\mu_{k_j} \frac{\partial f_{0,\mathbf{w},\mathbf{v}}^{k_j}}{\partial z_i}(\mathbf{b})$ and define $v_0\in\mathbb{N}$ and $I_0\subseteq \{1,\ldots,n\}$ as in \eqref{fvm}, that is, $v_0:=\mbox{min}\{v_i\, ;\, i\in I(\mathbf{w})\}$ and $I_0:=\{i\in I(\mathbf{w})\, ;\, v_i=v_0\}$ (note that, in general, $v_0\geq v_{\mbox{\tiny min}}$).
Then, as in the proof of Lemma \ref{Lemma1}, since $\lambda_0 \bar b_i\not=0$ and the set $\{i\in I(\mathbf{w})\, ;\, S_i\not=0\}$ is not empty (see Claim \ref{claim2}), we have $\delta=c+2v_0$ and $S_i\not=0$ for any $i\in I_0$.
In fact, for any $1\leq i\leq n$, the following holds:
\begin{equation}\label{cIII-2}
S_i\equiv\sum_{j=1}^m \tilde\mu_{k_j} \frac{\partial f_{0,\mathbf{w},\mathbf{v}}^{k_j}}{\partial z_i}(\mathbf{b})=
\left\{
\begin{aligned}
& \lambda_0 \bar b_i && \mbox{if} && i\in I_0,\\
& 0 && \mbox{if} && i\notin I_0.
\end{aligned}
\right.
\end{equation}
Since $I_0\not=\emptyset$ and $f_{0,\mathbf{w},\mathbf{v}}^{k_j}(\mathbf{b})=0$ ($1\leq j\leq m$), the relation \eqref{cIII-2} together with the Euler identity imply
\begin{equation}\label{eq-euler0427}
\begin{aligned}
0 & = \sum_{j=1}^{m}\tilde\mu_{k_j}\cdot d\big(\mathbf{v};f_{0,\mathbf{w}}^{k_j}\big) \cdot f_{0,\mathbf{w},\mathbf{v}}^{k_j}(\mathbf{b}) = \sum_{j=1}^{m}\tilde\mu_{k_j}\bigg(\sum_{i=1}^n v_i b_i \frac{\partial f_{0,\mathbf{w},\mathbf{v}}^{k_j}}{\partial z_i}(\mathbf{b})\bigg) \\
& = \sum_{i\in I_0}v_i b_i\bigg( \sum_{j=1}^{m}\tilde\mu_{k_j} \frac{\partial f_{0,\mathbf{w},\mathbf{v}}^{k_j}}{\partial z_i}(\mathbf{b}) \bigg)
=\lambda_0\cdot\sum_{i\in I_0} v_i |b_i|^2\not=0,
\end{aligned}
\end{equation}
which is a contradiction. This completes the proof of Lemma \ref{Lemma2} in the case $v_{\mbox{\tiny min}}>0$. 

Let us now assume $v_{\mbox{\tiny min}}=0$. Clearly, we still have $f_{0,\mathbf{w},\mathbf{v}}^{k_j}(\mathbf{b})=0$ for $1\leq j\leq m$.

\begin{claim}\label{claim2-20210426}
Even when $v_{\mbox{\tiny \emph{min}}}=0$, there exists $i_0\in I(\mathbf{w})$ such that 
$
\sum_{j=1}^m \tilde\mu_{k_j} \frac{\partial f_{0,\mathbf{w},\mathbf{v}}^{k_j}}{\partial z_{i_0}}(\mathbf{b})\not=0.
$
\end{claim}

\begin{proof}
When $v_{\mbox{\tiny min}}=0$, the argument given in the proof of Claim \ref{claim2} does not apply. In fact, in this case, Claim \ref{claim2-20210426} directly follows from Lemma \ref{Lemma1} and our choice of $\varepsilon_0$.
More precisely, we know that $\mathbf{b}\in\mathbb{C}^{*n}$, $f_{0,\mathbf{w},\mathbf{v}}^{k_j}(\mathbf{b})=0$ ($1\leq j\leq m$) and $f_{0,\mathbf{w},\mathbf{v}}^{k_j}=f_{0,\mathbf{v}+\nu\mathbf{w}}^{k_j}$ for $\nu\in\mathbb{N}$ large enough. Therefore, arguing by contradiction, if 
\begin{align*}
\sum_{j=1}^m \tilde\mu_{k_j} \frac{\partial f_{0,\mathbf{w},\mathbf{v}}^{k_j}}{\partial z_{i}}(\mathbf{b})=0
\end{align*}
 for all $i\in I(\mathbf{w})$ (and hence, by ($2'$), for all $1\leq i\leq n$), then Lemma \ref{Lemma1} and our choice of $\varepsilon_0$ show that
\begin{align*}
\mathbf{b}\notin\bigg\{\mathbf{z}\in \mathbb{C}^{*n}\, ;\, \sum_{i\in I(\mathbf{v}+\nu\mathbf{w})}|z_i|^2\leq \varepsilon_0^2\bigg\}. 
\end{align*}
However, since $I(\mathbf{v+\nu\mathbf{w}})\subseteq I(\mathbf{v})$, we have
\begin{align*}
\sum_{i\in I(\mathbf{v}+\nu\mathbf{w})}|b_i|^2  \leq 
\sum_{i\in I(\mathbf{v})}|b_i|^2 = \sum_{i\in I(\mathbf{v})}|a_i(0)|^2 
 \leq
\underbrace{\sum_{i\in I(\mathbf{w})}|a_i(0)|^2}_{\leq \varepsilon_0^2} + 
\underbrace{\sum_{i\in I(\mathbf{w})^c}|a_i(0)|^2}_{=0} \leq \varepsilon_0^2,
\end{align*}
which is a contradiction. (Here, $I(\mathbf{w})^c:=\{1,\ldots,n\}\setminus I(\mathbf{w})$.)
\end{proof}

Combined with ($2'$), Claim \ref{claim2-20210426} shows that $\lambda(s)$ is not constantly zero, and exactly as above we deduce that the relation \eqref{cIII-2} holds true for $v_{\mbox{\tiny min}}=0$ too. 
(The subset $I_0$ and the number $v_0$ are defined as before; we also use the same Laurent expansion $\lambda(s)=\lambda_0 s^c+\cdots$.)
If $v_0=0$, then $I_0=I(\mathbf{w})\cap I(\mathbf{v})=I(\mathbf{v}+\nu\mathbf{w})$, and since $\sum_{i\in I(\mathbf{v}+\nu\mathbf{w})}|b_i|^2\leq \varepsilon_0^2$, then, once again, we get a contradiction with Lemma \ref{Lemma1} and our choice of $\varepsilon_0$. If $v_0\not=0$, then we get a contradiction exactly as in \eqref{eq-euler0427}.
This completes the proof of Lemma \ref{Lemma2} in the case $v_{\mbox{\tiny min}}=0$.

\section{The ``non-family'' case}

In the previous section, we have studied the case of families of functions. Hereafter, we investigate the ``non-family'' case. For that purpose, we consider $2k_0$ non-constant polynomial functions $f^1(\mathbf{z}),\ldots, f^{k_0}(\mathbf{z})$ and $g^1(\mathbf{z}),\ldots, g^{k_0}(\mathbf{z})$, each of them in $n$ complex variables $\mathbf{z}=(z_1,\ldots,z_n)$, and as usual we assume $f^k(\mathbf{0})=g^k(\mathbf{0})=0$ for all $1\leq k\leq k_0$. 

\begin{assumptions}\label{ass-sect-mt}
Throughout this section, we suppose that the following two conditions hold true:
\begin{enumerate}
\item
for any $1\leq k\leq k_0$, the Newton boundaries $\Gamma(f^k)$ and $\Gamma(g^k)$ coincide;
\item
for any $k_1,\ldots,k_m\in \{k_1,\ldots,k_0\}$, the germs at $\mathbf{0}$ of the varieties $V(f^{k_1},\ldots,f^{k_m})$ and $V(g^{k_1},\ldots,g^{k_m})$ are the germs of non-degenerate complete intersection varieties.
\end{enumerate}
\end{assumptions}

Put $f(\mathbf{z}):=f^1(\mathbf{z})\cdots f^{k_0}(\mathbf{z})$ and $g(\mathbf{z}):=g^1(\mathbf{z})\cdots g^{k_0}(\mathbf{z})$. The second main theorem of this paper is stated as follows. Once more, note that when $k_0=1$, the functions $f$ and $g$ are non-degenerate, and then we recover Theorem 3 of \cite{O1}.

\begin{theorem}\label{mtfg}
Under Assumptions \ref{ass-sect-mt}, the Milnor fibrations of $f$ and $g$ at $\mathbf{0}$ are isomorphic.
\end{theorem}

\begin{proof}
For any $1\leq k\leq k_0$ and any $t\in D_1:=\{t\in\mathbb{C}\, ;\, \vert t\vert\leq 1\}$, we consider the polynomial functions 
\begin{equation*}
f_t^k(\mathbf{z}):=(1-t)f^k+t F^k
\quad\mbox{and}\quad
g_t^k(\mathbf{z}):=(1-t)g^k+t G^k,
\end{equation*} 
 where 
\begin{equation*}
F^k(\mathbf{z}):=\sum_{\alpha\in\Gamma(f^k)}c_\alpha\, \mathbf{z}^\alpha
\quad\mbox{and}\quad
G^k(\mathbf{z}):=\sum_{\alpha\in\Gamma(g^k)}c'_\alpha\, \mathbf{z}^\alpha
\end{equation*} 
are the Newton principal parts of 
\begin{equation*}
f^k(\mathbf{z}):=\sum_{\alpha\in\mathbb{N}^n}c_\alpha\, \mathbf{z}^\alpha
\quad\mbox{and}\quad
g^k(\mathbf{z}):=\sum_{\alpha\in\mathbb{N}^n}c'_\alpha\, \mathbf{z}^\alpha
\end{equation*} 
respectively.

\begin{claim}\label{c-imffpp}
The Milnor fibrations of $f^1\cdots f^{k_0}$ and  $F^1\cdots F^{k_0}$ (respectively, of $g^1\cdots g^{k_0}$ and  $G^1\cdots G^{k_0}$) at $\mathbf{0}$ are isomorphic.
\end{claim}

\begin{proof}
First, observe that for any $1\leq k\leq k_0$ and any positive weight vector $\mathbf{w}$, we have
\begin{equation*}
f_{t,\mathbf{w}}^k(\mathbf{z}):=((1-t)f^k+t F^k)_{\mathbf{w}}=(1-t)f^k_{\mathbf{w}}+t F^k_{\mathbf{w}}=f^k_{\mathbf{w}}.
\end{equation*} 
From this observation and Assumptions \ref{ass-sect-mt}, we deduce that $V(f_t^{k_1},\ldots,f_t^{k_m})$ is a non-degenera\-te complete intersection variety for any $t\in D_1$, and since $\Gamma(f^k_t)=\Gamma(f^k)$, it follows from Theorem 4.8 that the functions $f_t(\mathbf{z}):=f_t^1(\mathbf{z})\cdots f_t^{k_0}(\mathbf{z})$ and $f_0(\mathbf{z}):=f_0^1(\mathbf{z})\cdots f_0^{k_0}(\mathbf{z})=f^1(\mathbf{z})\cdots f^{k_0}(\mathbf{z})$ have isomorphic Milnor fibrations for any $t\in D_1$. In particular, taking $t=1$ gives that $F^1\cdots F^{k_0}$ and $f^1\cdots f^{k_0}$ have isomorphic Milnor fibrations as announced.
\end{proof}

\begin{claim}\label{c-mfppi}
The Milnor fibrations of $F^1\cdots F^{k_0}$ and $G^1\cdots G^{k_0}$ at $\mathbf{0}$ are isomorphic.
\end{claim}

The proof of this claim is given below. Of course, Theorem \ref{mtfg} follows from Claims \ref{c-imffpp} and \ref{c-mfppi}.
\end{proof}

Now, let us prove Claim \ref{c-mfppi}.

\begin{proof}[Proof of Claim \ref{c-mfppi}]
For each $1\leq k\leq k_0$, let $\nu_{k,1},\ldots,\nu_{k,n_k}$ be the integral points of $\Gamma(f^k)=\Gamma(F^k)$, and for any $\mathbf{c}_k=(c_{k,1},\ldots,c_{k,n_k})\in\mathbb{C}^{n_k}$ put
\begin{equation*}
h^k_{\mathbf{c}_k}(\mathbf{z}):=\sum_{j=1}^{n_k} c_{k,j}\, \mathbf{z}^{\nu_{k,j}}.
\end{equation*}
Now, consider the set $U$ of points $(\mathbf{c}_1,\ldots,\mathbf{c}_{k_0})$ in $\mathbb{C}^{n_1}\times \cdots\times \mathbb{C}^{n_{k_0}}$ such that:
\begin{enumerate}
\item
$\Gamma(h^k_{\mathbf{c}_k})=\Gamma(f^k)$ for any $1\leq k\leq k_0$;
\item
for any $1\leq k_1,\ldots,k_m\leq k_0$, the variety $V(h^{k_1}_{\mathbf{c}_{k_1}},\ldots,h^{k_m}_{\mathbf{c}_{k_m}})$ 
is a non-degenerate complete intersection variety.
\end{enumerate}

\begin{claim}\label{c-zo}
The set $U$ is a Zariski open subset of $\mathbb{C}^{n_1}\times \cdots\times \mathbb{C}^{n_{k_0}}$; in particular, it is path-connected.
\end{claim}

The special case $k_0=1$ in Claim \ref{c-zo} is treated in the appendix of \cite{O3}.
Before proving this claim in the general case, we complete the proof of Claim \ref{c-mfppi}.

For each $1\leq k\leq k_0$, let 
\begin{equation*}
\mathbf{c}_{k}(F^k):=(c_{k,1}(F^k),\ldots,c_{k,n_k}(F^k))
\quad\mbox{and}\quad
\mathbf{c}_{k}(G^k):=(c_{k,1}(G^k),\ldots,c_{k,n_k}(G^k))
\end{equation*}
 be the points defined by
\begin{equation*}
F^k(\mathbf{z}):=h^k_{\mathbf{c}_{k}(F^k)}(\mathbf{z})
:=\sum_{j=1}^{n_k} c_{k,j}(F^k)\, \mathbf{z}^{\nu_{k,j}}
\quad\mbox{and}\quad
G^k(\mathbf{z}):=h^k_{\mathbf{c}_{k}(G^k)}(\mathbf{z})
:=\sum_{j=1}^{n_k} c_{k,j}(G^k)\, \mathbf{z}^{\nu_{k,j}}.
\end{equation*}
By Claim \ref{c-zo}, we can choose a finite sequence of (say, $p_0$) $k_0$-tuples 
\begin{equation*}
(\mathbf{c}_1(1),\ldots,\mathbf{c}_{k_0}(1)), \ldots, 
(\mathbf{c}_1(p_0),\ldots,\mathbf{c}_{k_0}(p_0))
\end{equation*}
 in $U$, starting at $(\mathbf{c}_1(F^1),\ldots,\mathbf{c}_{k_0}(F^{k_0}))$ and ending at $(\mathbf{c}_1(G^1),\ldots,\mathbf{c}_{k_0}(G^{k_0}))$,
such that for each $1\leq p\leq p_0-1$, the straight-line segment 
\begin{equation*}
\ell_{p}(t):=(1-t)\, (\mathbf{c}_1(p),\ldots,\mathbf{c}_{k_0}(p))+t\, (\mathbf{c}_1(p+1),\ldots,\mathbf{c}_{k_0}(p+1))
\end{equation*}
($1\leq t\leq 1$) is contained in $U$. For each $1\leq p\leq p_0-1$, we consider the family $\big\{h_{\ell_{p}(t)}\big\}_{0\leq t\leq 1}$ of polynomial functions defined by
\begin{equation*}
h_{\ell_{p}(t)}(\mathbf{z}):=h^1_{\ell^1_{p}(t)}(\mathbf{z})\cdots 
h^{k_0}_{\ell^{k_0}_{p}(t)}(\mathbf{z}),
\end{equation*}
where $\ell^{k}_{p}(t):=(1-t)\mathbf{c}_{k}(p)+t\mathbf{c}_{k}(p+1)$ is the $k$th coordinate of $\ell_{p}(t)$.
By Theorem 4.8, the Milnor fibrations of $h_{\ell_{p}(0)}$ and $h_{\ell_{p}(1)}$ at $\mathbf{0}$ are isomorphic. Claim \ref{c-mfppi} then follows from the equalities
\begin{equation*}
h_{\ell_{1}(0)}=F^1\cdots F^{k_0}
\quad\mbox{and}\quad
h_{\ell_{p_0-1}(1)}=G^1\cdots G^{k_0}.
\end{equation*}
This completes the proof of Claim \ref{c-mfppi} (up to Claim \ref{c-zo}).
\end{proof}

Now let us prove Claim \ref{c-zo}.

\begin{proof}[Proof of Claim \ref{c-zo}]
For any $1\leq k\leq k_0$ and any positive weight vector $\mathbf{w}$ defining a (compact) face $\Delta(\mathbf{w};f^k)$ of $\Gamma(f^k)$ with maximal dimension, let us denote by $\theta_{k,1},\ldots,\theta_{k,q_k}$ the integral points of $\Delta(\mathbf{w};f^k)$. Then, for any $\mathbf{a}_k=(a_{k,1},\ldots,a_{k,q_k})\in\mathbb{C}^{q_k}$, put
\begin{equation*}
\phi^k_{\mathbf{a}_k}(\mathbf{z}):=\sum_{j=1}^{q_k} a_{k,j}\, \mathbf{z}^{\theta_{k,j}}.
\end{equation*}
Note that $\phi^k_{\mathbf{a}_k}$ is weighted homogeneous with respect to $\mathbf{w}$.
Now, consider the set $U_\mathbf{w}$ consisting of the points $(\mathbf{a}_1,\ldots,\mathbf{a}_{k_0})$ in $\mathbb{C}^{q_1}\times \cdots\times \mathbb{C}^{q_{k_0}}$ satisfying the following two properties:
\begin{enumerate}
\item
$\Gamma(\phi^k_{\mathbf{a}_k})=\Delta(\mathbf{w};f^k)$ for any $1\leq k\leq k_0$;
\item
for any $1\leq k_1,\ldots,k_m\leq k_0$, the variety $V(\phi^{k_1}_{\mathbf{a}_{k_1}},\ldots,\phi^{k_m}_{\mathbf{a}_{k_m}})$ is a non-degenerate complete intersection variety.
\end{enumerate}
To prove Claim \ref{c-zo}, it suffices to show that $U_\mathbf{w}$ is a Zariski open set. To do that, we first observe that since $\phi^{k_j}_{\mathbf{a}_{k_j}}$ ($1\leq j\leq m$) is weighted homogeneous, there exist $\lambda_1,\ldots,\lambda_n\in\mathbb{N}^*$ such that the polynomial
\begin{equation*}
\Phi^{k_j}_{\mathbf{a}_{k_j}}(z_1,\ldots,z_n):=\phi^{k_j}_{\mathbf{a}_{k_j}}(z_1^{\lambda_1},\ldots,z_n^{\lambda_n})
\end{equation*}
is homogeneous. Then, since $V(\phi^{k_1}_{\mathbf{a}_{k_1}},\ldots,\phi^{k_m}_{\mathbf{a}_{k_m}})$ is non-degenerate if and only if $V(\Phi^{k_1}_{\mathbf{a}_{k_1}},\ldots,\Phi^{k_m}_{\mathbf{a}_{k_m}})$ is non-degenerate, we may assume that $\phi^{k_j}_{\mathbf{a}_{k_j}}$ is homogeneous for any $1\leq j\leq m$.
Now, observe that for any positive weight vector $\mathbf{w}'$, the set $\Delta(\mathbf{w}';f^k_{\mathbf{w}})$ is a (compact) face of $\Delta(\mathbf{w};f^k)$, and then consider the set $V_{\mathbf{w}}(\mathbf{w}')$ made up of all the points $(\mathbf{a}_1,\ldots,\mathbf{a}_{k_0},\mathbf{z})$ in $\mathbb{P}^{q_1-1}\times \cdots\times \mathbb{P}^{q_{k_0}-1}\times \mathbb{P}^{n-1}$ for which there exists a subset $K\subseteq \{1,\ldots, k_0\}$ such that: 
\begin{equation*}
\forall k\in K,\ \phi^{k}_{\mathbf{a}_{k},\mathbf{w}'}(\mathbf{z})=0
\quad\mbox{and}\quad
\bigwedge_{k\in K} d\phi^{k}_{\mathbf{a}_{k},\mathbf{w}'}(\mathbf{z})=0,
\end{equation*}
where we still denote by $\mathbf{a}_1,\ldots,\mathbf{a}_{k_0}$ and $\mathbf{z}$ the classes of $\mathbf{a}_1,\ldots,\mathbf{a}_{k_0}$ and $\mathbf{z}$ in the projective spaces $\mathbb{P}^{q_1-1},\ldots,\mathbb{P}^{q_{k_0}-1}$ and $\mathbb{P}^{n-1}$ respectively.
(Once more, let us recall that $\phi^{k}_{\mathbf{a}_{k},\mathbf{w}'}\equiv (\phi^{k}_{\mathbf{a}_{k}})_{\mathbf{w}'}$ denotes the face function of $\phi^{k}_{\mathbf{a}_{k}}$ with respect to the weight vector $\mathbf{w}'$.) Let $\bar V_{\mathbf{w}}(\mathbf{w}')$ be the closure of $V_{\mathbf{w}}^*(\mathbf{w}'):=V_{\mathbf{w}}(\mathbf{w}')\cap \{z_1\cdots z_n\not=0\}$ in $\mathbb{P}^{q_1-1}\times \cdots\times \mathbb{P}^{q_{k_0}-1}\times \mathbb{P}^{n-1}$. Then $\bar V_{\mathbf{w}}(\mathbf{w}')$ is an algebraic set of dimension $\dim V^*(\mathbf{w}')$ (see \cite[Lemma 3.9]{W}). Let
\begin{equation*}
\pi\colon (\mathbb{P}^{q_1-1}\times \cdots\times \mathbb{P}^{q_{k_0}-1})\times \mathbb{P}^{n-1}\to \mathbb{P}^{q_1-1}\times \cdots\times \mathbb{P}^{q_{k_0}-1}
\end{equation*}
be the standard projection, and let
\begin{equation*}
W_{\mathbf{w}}^*:=\pi(V_{\mathbf{w}}^*)
\quad\mbox{and}\quad
\bar W_{\mathbf{w}}:=\pi(\bar V_{\mathbf{w}}),
\end{equation*}
where
\begin{equation*}
V_{\mathbf{w}}^*:=\bigcup_{\mathbf{w}'\in\mathbb{N}^{*n}} V_{\mathbf{w}}^*(\mathbf{w}')
\quad\mbox{and}\quad
\bar V_{\mathbf{w}}:=\bigcup_{\mathbf{w}'\in\mathbb{N}^{*n}} \bar V_{\mathbf{w}}(\mathbf{w}').
\end{equation*}
Clearly, $U_{\mathbf{w}}$ is the complement of $(p_1\times\cdots\times p_{k_0})^{-1}(W_{\mathbf{w}}^*)\cup \{\mathbf{0}\}$, where $p_k\colon \mathbb{C}^{q_k}\setminus \{\mathbf{0}\}\to \mathbb{P}^{q_k-1}$ is the standard canonical map. By the proper mapping theorem (see \cite[Satz 23]{Remmert}), $\bar W_{\mathbf{w}}$ is an algebraic set containing $W_{\mathbf{w}}^*$. In fact, we are going to prove that $W_{\mathbf{w}}^*=\bar W_{\mathbf{w}}$, which implies that $U_{\mathbf{w}}$ is a Zariski open set. To show the equality $W_{\mathbf{w}}^*=\bar W_{\mathbf{w}}$, we argue by contradiction. Suppose that $W_{\mathbf{w}}^*\subsetneq\bar W_{\mathbf{w}}$. Then there exists $(\mathbf{a}_1,\ldots,\mathbf{a}_{k_0},\mathbf{z})\in \bar V_{\mathbf{w}}$ such that $(\mathbf{a}_1,\ldots,\mathbf{a}_{k_0})\in \bar W_{\mathbf{w}}\setminus W_{\mathbf{w}}^*$. By the Curve Selection Lemma, there exists a real analytic curve
\begin{equation*}
\rho(s)=(\mathbf{a}_1(s),\ldots,\mathbf{a}_{k_0}(s),\mathbf{z}(s)),
\end{equation*}
$0\leq s\leq 1$, and a positive weight vector $\mathbf{w}'\in\mathbb{N}^{*n}$ such that $\rho(s)\in V^*_{\mathbf{w}}(\mathbf{w}')$ for $s>0$ and $\rho(0)=(\mathbf{a}_1,\ldots,\mathbf{a}_{k_0},\mathbf{z})$. For each $1\leq k\leq k_0$, write
\begin{equation*}
\mathbf{a}_k(s)=\mathbf{a}_k+\mathbf{a}_{k,1}s+\cdots
\quad\mbox{and}\quad
\mathbf{z}(s)=(b_1s^{w''_1}+\cdots,\ldots,b_ns^{w''_n}+\cdots).
\end{equation*}
By the assumption, $b_i\in\mathbb{C}^*$, $w''_i\in\mathbb{N}$ ($1\leq i\leq n$) and $\mbox{max}\{w''_i\, ;\, 1\leq i\leq n\}>0$. Moreover, for any $s\not=0$, there exists $K(s)\subseteq\{1,\ldots,k_0\}$ such that:
\begin{equation*}
\forall k\in K(s),\ \phi^{k}_{\mathbf{a}_{k}(s),\mathbf{w}'}(\mathbf{z}(s))=0
\quad\mbox{and}\quad
\bigwedge_{k\in K(s)} d\phi^{k}_{\mathbf{a}_{k}(s),\mathbf{w}'}(\mathbf{z}(s))=0.
\end{equation*}
By looking at the leading terms (with respect to $s$) in the above expressions, it follows that there exists a subset $K(0)\subseteq\{1,\ldots,k_0\}$ such that:
\begin{equation}\label{mt2f-tleq}
\forall k\in K(0),\ \big(\phi^{k}_{\mathbf{a}_{k},\mathbf{w}'}\big)_{\Delta}(\mathbf{b})=0
\quad\mbox{and}\quad
\bigwedge_{k\in K(0)} d\big(\phi^{k}_{\mathbf{a}_{k},\mathbf{w}'}\big)_{\Delta}(\mathbf{b})=0,
\end{equation}
where $\mathbf{b}:=(b_1,\ldots,b_n)$, $\Delta$ is the (compact) face of $\Delta(\mathbf{w}';f^k_{\mathbf{w}})$ on which the linear form \begin{equation*}
\alpha\in\Delta(\mathbf{w}';f^k_{\mathbf{w}})\mapsto \sum_{i=1}^n \alpha_i w_i''\in\mathbb{R}
\end{equation*}
takes its minimal value, and $\big(\phi^{k}_{\mathbf{a}_{k},\mathbf{w}'}\big)_{\Delta}$ is the corresponding face function. However, since $b_i\in\mathbb{C}^*$ for all $1\leq i\leq n$, the relations \eqref{mt2f-tleq} imply $(\mathbf{a}_1,\ldots,\mathbf{a}_{k_0})\in W^*_{\mathbf{w}}$, which is a contradiction.
\end{proof}

\bibliographystyle{amsplain}

\end{document}